\documentclass[12pt,twoside,a4paper]{article}
\usepackage{times}
\usepackage{bm}
\usepackage{braket}
\usepackage{graphicx}
\usepackage{theorem}
\usepackage{amsmath}
\usepackage{amssymb,mathrsfs}
\usepackage{latexsym}
\usepackage[numbers]{natbib}
\usepackage{algorithm}
\usepackage{algpseudocode}
\usepackage{caption}
\newfloat{procedure}{H}{loa}
\floatname{procedure}{Procedure} 
\usepackage[flushmargin]{footmisc}

\theorembodyfont{\itshape}
\newtheorem{thm}{Theorem}[section]
\newtheorem{lem}{Lemma}[section]
\newtheorem{prop}{Proposition}[section]
\newtheorem{coro}{Corollary}[section]

\newtheorem{cond}{Condition}

\theorembodyfont{\rmfamily}
{\bf}{\rm}
{\bf}{\rm}

\newtheorem{rem}{Remark}[section]{\itshape}{\rmfamily}
{\itshape}{\rmfamily}
\newenvironment{proof}{\noindent{\it Proof.~~}}{\medskip}

\makeatletter
\def\eqnarray{\stepcounter{equation}\let\@currentlabel=\theequation
\global\@eqnswtrue
\global\@eqcnt\z@\tabskip\@centering\let\\=\@eqncr
$$\halign to \displaywidth\bgroup\@eqnsel\hskip\@centering
  $\displaystyle\tabskip\z@{##}$&\global\@eqcnt\@ne 
  \hfil$\;{##}\;$\hfil
  &\global\@eqcnt\tw@ $\displaystyle\tabskip\z@{##}$\hfil 
   \tabskip\@centering&\llap{##}\tabskip\z@\cr}
\makeatother
\makeatletter
    \renewcommand{\theequation}{%
    \thesection.\arabic{equation}}
    \@addtoreset{equation}{section}
  \makeatother
\makeatletter
\def\Left#1#2\Right{\begingroup%
   \def\ts@r{\nulldelimiterspace=0pt \mathsurround=0pt}%
   \let\@hat=#1%
   \def\sht@im{#2}%
   \def\@t{{\mathchoice{\def\@fen{\displaystyle}\k@fel}%
          {\def\@fen{\textstyle}\k@fel}%
          {\def\@fen{\scriptstyle}\k@fel}%
          {\def\@fen{\scriptscriptstyle}\k@fel}}}%
   \def\g@rin{\ts@r\left\@hat\vphantom{\sht@im}\right.}%
   \def\k@fel{\setbox0=\hbox{$\@fen\g@rin$}\hbox{%
      $\@fen \kern.3875\wd0 \copy0 \kern-.3875\wd0%
      \llap{\copy0}\kern.3875\wd0$}}%
      \def\pt@h{\mathopen\@t}\pt@h\sht@im%
      \Right}%
\def\Right#1{\let\@hat=#1%
   \def\st@m{\mathclose\@t}%
   \st@m\endgroup}
\makeatother

\newcommand{\vc}{\bm}
\makeatletter
\DeclareRobustCommand\widecheck[1]{{\mathpalette\@widecheck{#1}}}
\def\@widecheck#1#2{%
    \setbox\z@\hbox{\m@th$#1#2$}%
    \setbox\tw@\hbox{\m@th$#1%
       \widehat{%
          \vrule\@width\z@\@height\ht\z@
          \vrule\@height\z@\@width\wd\z@}$}%
    \dp\tw@-\ht\z@
    \@tempdima\ht\z@ \advance\@tempdima2\ht\tw@ \divide\@tempdima\thr@@
    \setbox\tw@\hbox{%
       \raise\@tempdima\hbox{\scalebox{1}[-1]{\lower\@tempdima\box
\tw@}}}%
    {\ooalign{\box\tw@ \cr \box\z@}}}
\makeatother

\newcommand{\ol}{\overline}

\newcommand{\wt}{\widetilde}
\newcommand{\wh}{\widehat}
\newcommand{\bv}{\breve}

\def\simhm#1{\stackrel{#1}{\sim}}

\newcommand{\vertiii}[1]%
{{\left\vert\kern-0.25ex\left\vert\kern-0.25ex\left\vert #1 
 \right\vert\kern-0.25ex\right\vert\kern-0.25ex\right\vert}}
\newcommand{\down}[2]{\smash{\lower#1\hbox{#2}}}
\newcommand{\up}[2]{\smash{\lower-#1\hbox{#2}}}

\newcommand{\dm}{\displaystyle}
\newcommand{\qed}{\hspace*{\fill}$\Box$}

\def\presub#1{\hspace{0.05em}{}_{#1}\hspace{-0.05em}}
\newcommand{\vmax}{\vee}

\newcommand{\hms}{\hspace{0.1em}}

\newcommand{\PP}{\mathsf{P}}

\newcommand{\bbB}{\mathbb{B}}
\newcommand{\bbC}{\mathbb{C}}

\newcommand{\bbL}{\mathbb{L}}
\newcommand{\bbM}{\mathbb{M}}
\newcommand{\bbN}{\mathbb{N}}
\newcommand{\bbR}{\mathbb{R}}
\newcommand{\bbS}{\mathbb{S}}

\newcommand{\bbZ}{\mathbb{Z}}
\DeclareMathOperator*{\argmin}{arg\,min}

\newcommand{\rd}{{\rm d}}
\newcommand{\re}{{\rm e}}

\renewcommand{\labelenumi}{(\roman{enumi})}
\newcommand{\dd}[1]{\if#11 1\!\!1 
\else {\if#1C I\!\!\!C
\else {\if#1G I\!\!\!G 
\else {\if#1J J\!\!\!J 
\else {\if#1S S\!\!\!S
\else {\if#1Z Z\!\!\!Z
\else {\if#1Q O\!\!\!\!Q
\else I\!\!#1
\fi}
\fi}
\fi}
\fi} 
\fi} 
\fi} 
\fi} 

\makeatother

\begin{document}\thispagestyle{empty} 

\hfill

\vspace{-10mm}

{\Large{\bf
\begin{center}
A sequential update algorithm for computing\\
the stationary distribution vector in upper block-Hessenberg Markov chains\footnote[1]{%
Published online in {\it Queueing Systems} on February 21, 2019 (doi: 10.1007/s11134-019-09599-x)
}
%
%
%
\end{center}
}
}

\begin{center}
{
Hiroyuki Masuyama%
\footnote[2]{E-mail: masuyama@sys.i.kyoto-u.ac.jp}
}

\medskip

{\small
Department of Systems
Science, Graduate School of Informatics, Kyoto University\\
Kyoto 606-8501, Japan
}

\bigskip
\medskip

{\small
\textbf{Abstract}

\medskip

\begin{tabular}{p{0.85\textwidth}}
This paper proposes a new algorithm for computing the stationary distribution vector in continuous-time upper block-Hessenberg Markov chains. To this end, we consider the last-block-column-linearly-augmented (LBCL-augmented) truncation of the (infinitesimal) generator of the upper block-Hessenberg Markov chain. The LBCL-augmented truncation is a linearly-augmented truncation such that the augmentation distribution has its probability mass only on the last block column. We first derive an upper bound for the total variation distance between the respective stationary distribution vectors of the original generator and its LBCL-augmented truncation. Based on the upper bound, we then establish a series of linear fractional programming (LFP) problems to obtain augmentation distribution vectors such that the bound converges to zero. Using the optimal solutions of the LFP problems, we construct a matrix-infinite-product (MIP) form of the original (i.e., not approximate) stationary distribution vector and develop a sequential update algorithm for computing the MIP form. Finally, we demonstrate the applicability of our algorithm to BMAP/M/$\infty$ queues and M/M/$s$ retrial queues.
\end{tabular}
}
\end{center}

\begin{center}
\begin{tabular}{p{0.90\textwidth}}
{\small

\medskip

}
\end{tabular}

\end{center}

\section{Introduction}\label{sec-intro}

This paper considers an upper block-Hessenberg Markov chain in continuous time. To describe such a Markov chain, we first introduce some symbols. Let $\bbR_+$ denote the set of all nonnegative real numbers, i.e., $\bbR_+=[0,\infty)$. Let $\bbN = \{1,2,3,\dots\}$, $\bbZ_+ = \{0,1,2,\dots\}$, and $\bbZ_n = \{0,1,\dots,n\}$ for $n \in \bbZ_+$.
We then introduce some sets of pairs of integers:
\begin{eqnarray*}
\bbS &=& \bigcup_{k=0}^{\infty} \bbL_k, \quad
\bbS_n = \bigcup_{k=0}^n \bbL_k, \quad
\ol{\bbS}_n = \bbS \setminus \bbS_n,
\quad 
n \in \bbZ_+,
\\
\bbL_k 
&=& \{k\} \times \bbM_k, \quad k \in \bbZ_+,
\end{eqnarray*}
where $\bbM_k = \{1,2,\dots,M_k\} \subset \bbN$.
We also define $(k,i;\ell,j)$ as
an ordered pair $((k,i),(\ell,j))$ in $\bbS^2$. Furthermore, we define $\vc{e}=(1,1,\dots)^{\top}$, which has an appropriate (finite or infinite) number of ones.

Let $\{(X(t),J(t));t \in \bbR_+\}$ denote a regular-jump bivariate Markov chain with state space $\bbS$ (see \cite[Chapter 8, Definition 2.5]{Brem99} for the definition of regular-jump Markov chains). Let $\vc{Q}:=(q(k,i;\ell,j))_{(k,i;\ell,j) \in\bbS^2}$ denote the (infinitesimal) generator of the Markov chain $\{(X(t),J(t))\}$, which is in an upper block-Hessenberg form:
\begin{equation}
\vc{Q} = 
\bordermatrix{
        & \bbL_0 & \bbL_1    & \bbL_2    & \bbL_3   & \cdots
\cr
\bbL_0 		& 
\vc{Q}_{0,0}  	& 
\vc{Q}_{0,1} 	& 
\vc{Q}_{0,2}  	& 
\vc{Q}_{0,3} 	&  
\cdots
\cr
\bbL_1 		&
\vc{Q}_{1,0}  	& 
\vc{Q}_{1,1} 	& 
\vc{Q}_{1,2}  	& 
\vc{Q}_{1,3} 	&  
\cdots
\cr
\bbL_2 		& 
\vc{O}  		& 
\vc{Q}_{2,1} 	& 
\vc{Q}_{2,2}  	& 
\vc{Q}_{2,3} 	&  
\cdots
\cr
\bbL_3 & 
\vc{O}  		& 
\vc{O}  		& 
\vc{Q}_{3,2}  	& 
\vc{Q}_{3,3} 	&  
\cdots
\cr
~\vdots  	& 
\vdots     		& 
\vdots     		&  
\vdots    		& 
\vdots    		& 
\ddots
}.
\label{defn-Q-MG1-type}
\end{equation}
We refer to $\{(X(t),J(t))\}$ as the {\it upper block-Hessenberg Markov chain} (which may be called the {\it level-dependent M/G/1-type Markov chain}) and refer to $X(t)$ and $J(t)$ as the {\it level variable} and the {\it phase variable}, respectively. Note that if $(X(t),J(t)) \in \bbL_k$ then $X(t) = k$ and thus $\bbL_k$ is called {\it level} $k$.

Throughout the paper, unless otherwise stated, we assume that $\{(X(t),J(t))\}$ is ergodic (i.e., irreducible, aperiodic and positive recurrent). We then define $\vc{\pi}:=(\pi(k,i))_{(k,i)\in\bbS} > \vc{0}$ as the unique stationary
distribution vector of the ergodic generator $\vc{Q}$ (see, e.g., \cite[Chapter 5, Theorems 4.4 and 4.5]{Ande91}). By definition, $\vc{\pi}\vc{Q} = \vc{0}$ and $\vc{\pi}\vc{e}=1$.
For later use, we also define $\vc{\pi}_k =(\pi(k,i))_{i\in\bbM_k}$ for $k \in \bbZ_+$ and partition $\vc{\pi}$ as
\begin{eqnarray*}
\vc{\pi}
&=& 
\bordermatrix{
        & 
\bbL_0 & 
\bbL_1 & 
\cdots
\cr
			& 
\vc{\pi}_0  & 
\vc{\pi}_1 	& 
\cdots
}.
\end{eqnarray*}

It is, in general, difficult to obtain an explicit expression of $\vc{\pi}=(\vc{\pi}_0,\vc{\pi}_1,\dots)$. 
Thus, we study the computation of the stationary
distribution vector $\vc{\pi}$ through a {\it linearly augmented truncation} of the ergodic generator $\vc{Q}$.
The linearly augmented truncation is described below.

Let $\presub{(n)}\vc{Q}:=(\presub{(n)}q(k,i;\ell,j))_{(k,i;\ell,j)\in (\bbS_n)^2}$, $n \in \bbZ_+$, denote the northwest corner truncation of the ergodic generator $\vc{Q}$, which is given by
\begin{eqnarray*}
\presub{(n)}\vc{Q}
=
\left(
\begin{array}{ccccccc}
\vc{Q}_{0,0} 	&
\vc{Q}_{0,1} 	&
\vc{Q}_{0,2} 	&
\cdots		&
\vc{Q}_{0,n-2} 	&
\vc{Q}_{0,n-1} 	&
\vc{Q}_{0,n} 	
\\
\vc{Q}_{1,0} 	&
\vc{Q}_{1,1} 	&
\vc{Q}_{1,2} 	&
\cdots		&
\vc{Q}_{1,n-2} 	&
\vc{Q}_{1,n-1} 	&
\vc{Q}_{1,n} 	
\\
\vc{O} 		&
\vc{Q}_{2,1} 	&
\vc{Q}_{2,2} 	&
\cdots		&
\vc{Q}_{2,n-2} 	&
\vc{Q}_{2,n-1} 	&
\vc{Q}_{2,n} 	
\\
\vdots		&
\vdots		&
\vdots		&
\ddots		&
\vdots		&
\vdots		&
\vdots		
\\
\vc{O}		&
\vc{O}		&
\vc{O}		&
\cdots		&
\vc{Q}_{n-1,n-2} 	&
\vc{Q}_{n-1,n-1} 	&
\vc{Q}_{n-1,n} 
\\
\vc{O}		&
\vc{O}		&
\vc{O}		&
\cdots		&
\vc{O}		&
\vc{Q}_{n,n-1} 	&
\vc{Q}_{n,n} 
\\
\end{array}
\right).
\end{eqnarray*}
We then define $\presub{(n)}\ol{\vc{Q}}:=(\presub{(n)}\ol{q}(k,i;\ell,j))_{(k,i;\ell,j)\in(\bbS_n)^2}$, $n\in\bbZ_+$, as a $Q$-matrix (diagonally dominant matrix with nonpositive diagonal elements and nonnegative off-diagonal ones; see, e.g., \cite[Section~2.1]{Ande91}) such that
\begin{equation}
\presub{(n)}\ol{\vc{Q}}
= \presub{(n)}\vc{Q} 
- \presub{(n)}\vc{Q}\vc{e} \presub{(n)}\vc{\alpha},
\qquad n\in\bbZ_+,
\label{defn-(n)ol{Q}} 
\end{equation}
where $\presub{(n)}\vc{\alpha}:=(\presub{(n)}\alpha(k,i))_{(k,i) \in \bbS_n}$ is a probability vector. We refer to $\presub{(n)}\ol{\vc{Q}}$ as the {\it linearly augmented truncation} of $\vc{Q}$. We also refer to $\presub{(n)}\vc{\alpha}$ as the {\it augmentation distribution vector}.

Let $\presub{(n)}\ol{\vc{\pi}}=(\presub{(n)}\ol{\pi}(k,i))_{(k,i) \in \bbS_n}$, $n\in\bbZ_+$, denote
\begin{equation}
\presub{(n)}\ol{\vc{\pi}} 
= {\presub{(n)}\vc{\alpha} (- \presub{(n)}\vc{Q})^{-1}
\over \presub{(n)}\vc{\alpha} (- \presub{(n)}\vc{Q})^{-1}\vc{e}},
\qquad n \in \bbZ_+,
\label{defn-(n)ol{pi}}
\end{equation}
where $(- \presub{(n)}\vc{Q})^{-1}$ exists due to the ergodicity of $\vc{Q}$. From (\ref{defn-(n)ol{Q}}) and (\ref{defn-(n)ol{pi}}), we have
\[
\presub{(n)}\ol{\vc{\pi}}\presub{(n)}\ol{\vc{Q}}=\vc{0},
\quad 
\presub{(n)}\ol{\vc{\pi}} \ge \vc{0},\quad \presub{(n)}\ol{\vc{\pi}}\vc{e}=1;
\]
that is,
$\presub{(n)}\ol{\vc{\pi}}$ is a stationary distribution vector of the linearly augmented truncation $\presub{(n)}\ol{\vc{Q}}$. Furthermore, as $n\to\infty$, each element of $\presub{(n)}\ol{\vc{Q}}$ converges to the corresponding one of $\vc{Q}$. Thus, we can expect $\presub{(n)}\ol{\vc{\pi}}$ to be an approximation to $\vc{\pi}$. 
This is why we refer to $\presub{(n)}\ol{\vc{\pi}}$ as the {\it linearly augmented truncation approximation} to $\vc{\pi}$. 

We note that if the augmentation distribution vector $\presub{(n)}\vc{\alpha}$ has its probability mass only on the last block (i.e., $\bbL_n$) then the linearly augmented truncation $\presub{(n)}\ol{\vc{Q}}$ inherits upper block-Hessenberg structure from the original generator $\vc{Q}$. 
To utilize this tractable structure, we focus on a special linearly augmented truncation $\presub{(n)}\ol{\vc{Q}}$ with $\presub{(n)}\vc{\alpha} = \presub{(n)}\wh{\vc{\alpha}}$, where $\presub{(n)}\wh{\vc{\alpha}}$ is a probability vector such that
\begin{equation}
\presub{(n)}\wh{\vc{\alpha}}
=
\bordermatrix{
        	& 
\bbS_{n-1} &
\bbL_n 
\cr
			& 
\vc{0}  	& 
\vc{\alpha}_n	
}.
\label{cond-(n_s)alpha}
\end{equation}
For convenience, we refer to such a linearly augmented truncation as a {\it last-block-column-linearly-augmented truncation (LBCL-augmented truncation)}. 

We now define $\presub{(n)}\wh{\vc{Q}}:=(\presub{(n)}\wh{q}(k,i;\ell,j)_{(k,i;\ell,j) \in (\bbS_n)^2}$, $n\in\bbZ_+$, as the LBCL-augmented truncation of $\vc{Q}$, that is, a $Q$-matrix such that
\begin{equation}
\presub{(n)}\wh{\vc{Q}}
= \presub{(n)}\vc{Q} 
- \presub{(n)}\vc{Q}\vc{e} \presub{(n)}\wh{\vc{\alpha}},
\qquad n\in\bbZ_+.
\label{defn-(n)wh{Q}}
\end{equation}
We also define $\presub{(n)}\wh{\vc{\pi}}:=(\presub{(n)}\wh{\pi}(k,i))_{(k,i) \in \bbS_n}$, $n\in\bbZ_+$, as
\begin{equation}
\presub{(n)}\wh{\vc{\pi}} 
= {\presub{(n)}\wh{\vc{\alpha}} (- \presub{(n)}\vc{Q})^{-1}
\over \presub{(n)}\wh{\vc{\alpha}} (- \presub{(n)}\vc{Q})^{-1}\vc{e}},
\qquad n \in \bbZ_+.
\label{defn-(n)ol{pi}_n}
\end{equation}
Note that $\presub{(n)}\wh{\vc{\pi}}$ is equal to $\presub{(n)}\ol{\vc{\pi}}$ in (\ref{defn-(n)ol{pi}}) with $\presub{(n)}\vc{\alpha} = \presub{(n)}\wh{\vc{\alpha}}$; that is, $\presub{(n)}\wh{\vc{\pi}}$ is a stationary distribution vector of the LBCL-augmented truncation $\presub{(n)}\wh{\vc{Q}}$. Hence, we call $\presub{(n)}\wh{\vc{\pi}}$ the {\it last-block-column-linearly-augmented truncation approximation (LBCL-augmented truncation approximation)} to $\vc{\pi}$.

In this paper, we propose a new algorithm for computing the original stationary distribution vector $\vc{\pi}$ by using the LBCL-augmented truncation approximation $\presub{(n)}\wh{\vc{\pi}}$. In fact, $\presub{(n)}\wh{\vc{\pi}}$ does not necessarily converge to $\vc{\pi}$ as $n \to \infty$ (see Section~\ref{subsec-example}). We solve such a problem by choosing $\presub{(n)}\wh{\vc{\alpha}}$ adaptively for each $n \in \bbZ_+$. To achieve this, we first derive an upper bound for the total variation distance between $\vc{\pi}$ and $\presub{(n)}\widehat{\vc{\pi}}$. With this upper bound, we establish a series of linear fractional programming (LFP) problems for finding $\{\presub{(n)}\wh{\vc{\alpha}};n\in\bbZ_+\}$ such that $\{\presub{(n)}\wh{\vc{\pi}};n\in\bbZ_+\}$ converges to $\vc{\pi}$. Fortunately, the optimal solutions of the LFP problems are explicitly obtained. Thus, we can readily construct a convergent sequence of LBCL-augmented truncation approximations, which yields a matrix-infinite-product (MIP) form of $\vc{\pi}$. We note that the LFP problems are not given in advance but are formulated successively while constructing the MIP form. As a result, we can develop a sequential update algorithm for computing $\vc{\pi}$. 

We now review related work. Some researchers \cite{Baum12-Procedia,Brig95,Phun10-QTNA} have studied the computation of level-dependent quasi-birth-and-death processes (LD-QBDs), which belong to a special case of upper block-Hessenberg Markov chains. These previous studies propose algorithms for computing the {\it conditional stationary distribution vector} $\vc{\pi}^{(N)}$:
\[
\vc{\pi}^{(N)} 
= 
{(\vc{\pi}_0,\vc{\pi}_1,\dots,\vc{\pi}_N) 
\over 
\sum_{\ell=0}^N \vc{\pi}_{\ell}},
\]
where $N \in \bbN$ is the truncation parameter that should be determined so that $\vc{\pi}^{(N)}$ is sufficiently close to $\vc{\pi}$.
Takine~\cite{Taki16} develops an algorithm for computing $\vc{\pi}^{(N)}$ of a special upper block-Hessenberg Markov chain, which assumes that, for all sufficiently
large $n\in\bbZ_+$, the $\vc{Q}_{n,n-1}$ are nonsingular and the $\vc{Q}_{n,n}$ are of the same order (see Assumption 1 therein). These additional assumptions in \cite{Taki16} are removed by Kimura and Takine \cite{M.Kimu18}. Besides, Shin and Pearce~\cite{Shin98}, Li et al.~\cite{LiQuan05}, and Klimenok and Dudin~\cite{Klim06} modify transition rates (or transition probabilities) such that they are eventually level independent, and then these researchers establish algorithms for computing approximately the stationary distribution vectors of upper block-Hessenberg Markov chains.

The algorithms proposed in \cite{Brig95,Shin98} have update procedures to improve their outputs, like our algorithm. However, their update procedures need to recompute, from scratch, most components of their new outputs every time. On the other hand, our algorithm utilizes the components of the current result, together with some additional computation, to generate an updated result. This is a remarkable feature of our algorithm.

The rest of this paper is divided into four sections. Section~\ref{sec-MIP-form-solution} describes preliminary results on the LBCL-augmented truncation approximation for upper block-Hessenberg Markov chains. Section~\ref{sec-algorithm} proposes a sequential update algorithm that generates a sequence of LBCL-augmented truncation approximations converging to the original stationary distribution vector. Section~\ref{sec-discussion} demonstrates the applicability of the proposed algorithm. Finally, Section~\ref{sec-remarks} provides concluding remarks.


\section{The LBCL-augmented truncation approximation}\label{sec-MIP-form-solution}

This section consists of three subsections. 
In Section~\ref{subsec-matrix-product-form}, we show a matrix-product form of the LBCL-augmented truncation approximation $\presub{(n)}\wh{\vc{\pi}}$. In Section~\ref{subsec-error-bound}, we derive an error bound for $\presub{(n)}\wh{\vc{\pi}}$, more specifically, an upper bound for the total variation distance between $\presub{(n)}\wh{\vc{\pi}}$ and $\vc{\pi}$.
In Section~\ref{subsec-example}, we provide an example such that $\presub{(n)}\wh{\vc{\pi}}$ does not converge to $\vc{\pi}$ as $n \to \infty$. 

Before entering the body of this section, we describe our notation.
For any matrix $\vc{M}$ (resp.~vector $\vc{m}$), let $|\vc{M}|$ (resp.~$|\vc{m}|$) denote the matrix (resp.~vector) obtained by taking the absolute value of each element of $\vc{M}$ (resp.~$\vc{m}$). A finite matrix is treated, if necessary, as an infinite matrix that keeps the existing elements in their original positions and has an infinite number of zeros in the other positions. Such treatment is also applied to finite vectors. Thus, for example, it follows from (\ref{defn-(n)wh{Q}}) that
\begin{eqnarray}
\lefteqn{
\presub{(n)}\wh{\vc{\pi}} \,| \presub{(n)}\wh{\vc{Q}} - \vc{Q} |
}
\quad &&
\nonumber
\\
&=& \bordermatrix{
        	& 
\bbS_n 		& 	
\ol{\bbS}_n 	
\cr
			& 
\presub{(n)}\wh{\vc{\pi}} &
\vc{0}  	
}
\nonumber
\\
&& {} \times 
\left|
\bordermatrix{
        	& 
\bbS_n 		& 	
\ol{\bbS}_n 	
\cr
\bbS_n &
\presub{(n)}\vc{Q}
-\presub{(n)}\vc{Q} \vc{e} \presub{(n)}\wh{\vc{\alpha}} &
\vc{O}
\cr
\ol{\bbS}_n &
\vc{O} & \vc{O}
}
-
\bordermatrix{
        & \bbS_n & \ol{\bbS}_n
\cr
\bbS_n 		& 
\presub{(n)}\vc{Q} 	& \presub{(n)}\vc{Q}_{>n}	 
\cr
\ol{\bbS}_n		&
\ast  	& 
\ast	
}
\right|
\nonumber
\\
&=& 
\left(
\begin{array}{c@{~}|@{~}c}
\presub{(n)}\wh{\vc{\pi}} & \vc{0}
\end{array}
\right)
\left(
\begin{array}{c@{~}|@{~}c}
-\presub{(n)}\vc{Q} \vc{e} \presub{(n)}\wh{\vc{\alpha}} &
\presub{(n)}\vc{Q}_{>n}
\rule[-2.5mm]{0mm}{2mm}{}
\\
\hline
* & *
\end{array}
\right),
\label{eqn-(n)wh{pi}((n)wh{Q}-Q)}
\end{eqnarray}
where $\presub{(n)}\vc{Q}_{>n} = (q(k,i;\ell,j))_{(k,i;\ell,j) \times \bbS_n \times \ol{\bbS}_n}$.
It also follows from (\ref{cond-(n_s)alpha}) that, for any column vector $\vc{v}:= (v(k,i))_{(k,i)\in\bbS}$, 
\begin{equation}
\presub{(n)}\wh{\vc{\alpha}}\vc{v}
= \bordermatrix{
        	& 
\bbS_n 	& 	
\ol{\bbS}_n 	
\cr
	& 
\presub{(n)}\wh{\vc{\alpha}} &
\vc{0}  	
}
\vc{v}
= \bordermatrix{
        	& 
\bbS_{n-1} 	& 	
\bbL_n 		&
\ol{\bbS}_n 	
\cr
			& 
\vc{0}  	& 
\vc{\alpha}_n &
\vc{0}  	
}
\vc{v}
= \vc{\alpha}_n\vc{v}_n,
\label{eqn-(n)wh{alpha}v}
\end{equation}
where $\vc{v}_n = (v(k,i))_{(k,i) \in\bbL_n} = (v(n,i))_{i\in\bbM_n}$ for $n \in \bbZ_+$. Furthermore, we use the following notation: If a sequence $\{\vc{Z}_n; n \in \bbZ_+\}$ of finite matrices (or vectors) converges element-wise to an infinite matrix (or vector) $\vc{Z}$, then we denote this convergence by $\lim_{n\to\infty}\vc{Z}_n = \vc{Z}$. We also define the empty sum as zero (e.g., $\sum_{k=1}^0 \cdot = 0$).

\subsection{A matrix-product form}\label{subsec-matrix-product-form}

We partition $\presub{(n)}\wh{\vc{\pi}}$ and $(-\presub{(n)}\vc{Q})^{-1}$ level-wise as follows:
\begin{align}
\presub{(n)}\wh{\vc{\pi}}
&=
\bordermatrix{
        & 
\bbL_0 	& 
\bbL_1 	& 
\cdots	&
\bbL_n 	
\cr
			& 
\presub{(n)}\wh{\vc{\pi}}_{0} & 
\presub{(n)}\wh{\vc{\pi}}_{1}	& 
\cdots							&
\presub{(n)}\wh{\vc{\pi}}_{n}	 
}, & n &\in \bbZ_+,
\nonumber
\\
(-\presub{(n)}\vc{Q})^{-1} 
&=
\bordermatrix{
        			& 
\bbL_0 			& 
\bbL_1 			& 
\cdots  			& 
\bbL_n 
\cr
\bbL_0			& 
\presub{(n)}\vc{X}_{0,0}  	& 
\presub{(n)}\vc{X}_{0,1}  	& 
\cdots			&
\presub{(n)}\vc{X}_{0,n}  	 
\cr
\bbL_1		 	& 
\presub{(n)}\vc{X}_{1,0}  	& 
\presub{(n)}\vc{X}_{1,1}  	& 
\cdots			&
\presub{(n)}\vc{X}_{1,n}  	 
\cr
~\vdots			&
\vdots			&
\vdots			&
\ddots			&
\vdots						
\cr
\bbL_n			& 
\presub{(n)}\vc{X}_{n,0}  	& 
\presub{(n)}\vc{X}_{n,1}  	& 
\cdots			&
\presub{(n)}\vc{X}_{n,n}  	 
}, & n &\in \bbZ_+. 
\label{defn-(-Q)^{-1}-MG1-type}
\end{align}
From (\ref{defn-(-Q)^{-1}-MG1-type}) and (\ref{cond-(n_s)alpha}), we have
\begin{equation}
\presub{(n)}\wh{\vc{\alpha}}
(-\presub{(n)}\vc{Q})^{-1}
=\vc{\alpha}_n 
(\presub{(n)}\vc{X}_{n,0},\presub{(n)}\vc{X}_{n,1},\dots,\presub{(n)}\vc{X}_{n,n}),\qquad n \in \bbZ_+.
\label{eqn-(n)wh{alpha}(-(n)Q)^{-1}}
\end{equation}
Substituting (\ref{eqn-(n)wh{alpha}(-(n)Q)^{-1}}) into (\ref{defn-(n)ol{pi}_n}) yields
\[
\presub{(n)}\wh{\vc{\pi}}
=
{\vc{\alpha}_n 
(\presub{(n)}\vc{X}_{n,0},\presub{(n)}\vc{X}_{n,1},\dots,\presub{(n)}\vc{X}_{n,n})
\over
\vc{\alpha}_n \sum_{\ell=0}^n \presub{(n)}\vc{X}_{n,\ell}\vc{e}
},\qquad n \in \bbZ_+,
\]
which leads to
\begin{eqnarray}
\presub{(n)}\wh{\vc{\pi}}_{k}
=
{\vc{\alpha}_n 
\presub{(n)}\vc{X}_{n,k}
\over
\vc{\alpha}_n \sum_{\ell=0}^n \presub{(n)}\vc{X}_{n,\ell}\vc{e}
}, \qquad n \in \bbZ_+,\ k \in \bbZ_n.
\label{eqn-(n)pi_{n,k}}
\end{eqnarray}
Note that, because $(-\presub{(n)}\vc{Q})^{-1}\presub{(n)}\vc{Q}=-\vc{I}$, the inverse matrix $(-\presub{(n)}\vc{Q})^{-1} \ge \vc{O}$ has no zero rows. Therefore,
\begin{eqnarray}
\sum_{\ell=0}^n \presub{(n)}\vc{X}_{n,\ell}\vc{e} &>& \vc{0},
\qquad n \in \bbZ_+,\ \ell \in \bbZ_n.
\label{ineqn-(n)X_{n,l}e>0}
\end{eqnarray}

We derive a matrix-product form of $\presub{(n)}\wh{\vc{\pi}}_{k}$, $k\in\bbZ_n$, from (\ref{eqn-(n)pi_{n,k}}). To do this, we need some preparation. We first partition $\presub{(n)}\vc{Q}$ as
\[
\presub{(n)}\vc{Q}
= \left(
\begin{array}{cccc|c}
   						   &         			 	&  &      &	\,\vc{Q}_{0,n}
\\
						   &    & \presub{(n-1)}\vc{Q} &      &  \,\vdots
\\
\rule[-2mm]{0mm}{4mm}	   &                    	&  &      &	\,\vc{Q}_{n-1,n}
\\
\hline
\rule{0mm}{4mm}\vc{O} & \cdots & \vc{O} & \vc{Q}_{n,n-1} & \, \vc{Q}_{n,n}
\end{array}
\right),\qquad n \in \bbN.
\]
From this equation and (\ref{defn-(-Q)^{-1}-MG1-type}), we have the following  (see the last two equations in \cite[Section~0.7.3]{Horn13}): For $n \in \bbN$,
\begin{eqnarray}
\presub{(n)}\vc{X}_{n,n}
&=&
\left[
- \vc{Q}_{n,n}
- (\vc{O},\dots,\vc{O},\vc{Q}_{n,n-1})
(-\presub{(n-1)}\vc{Q})^{-1}
\left(
\begin{array}{c}
\vc{Q}_{0,n} 
\\
\vc{Q}_{1,n} 
\\
\vdots
\\
\vc{Q}_{n-1,n} 
\end{array}
\right)
\right]^{-1}
\nonumber
\\
&=& \left( 
- \vc{Q}_{n,n} 
- \vc{Q}_{n,n-1} \sum_{\ell=0}^{n-1} \presub{(n-1)}\vc{X}_{n-1,\ell}\vc{Q}_{\ell,n} \right)^{-1},
\label{defn-(n)X_{n,n}}
\end{eqnarray}
and
\begin{eqnarray}
\presub{(n)}\vc{X}_{n,k}
&=& \presub{(n)}\vc{X}_{n,n} \cdot (\vc{O},\dots,\vc{O},\vc{Q}_{n,n-1})
(-\presub{(n-1)}\vc{Q})^{-1}
\nonumber
\\
&=& \presub{(n)}\vc{X}_{n,n} \cdot \vc{Q}_{n,n-1}\presub{(n-1)}\vc{X}_{n-1,k},
\qquad k \in \bbZ_{n-1}.
\label{defn-(n)X_{n,k}}
\end{eqnarray}
We also define
\begin{align}
\vc{U}_n^*
&= \presub{(n)}\vc{X}_{n,n},
\label{eqn-U_k^*=(n)X_{n,n}}
\\
\vc{U}_{n,k}
&= 
\left\{
\begin{array}{ll}
\vc{Q}_{n,n-1}\presub{(n-1)}\vc{X}_{n-1,k}, & \quad k \in \bbZ_{n-1},
\\
\vc{I},                                     & \quad k=n,
\end{array}
\right.
\label{defn-U_{n,k}}
\end{align}
for $n \in \bbZ_+$. It then follows from (\ref{defn-(n)X_{n,n}}), (\ref{eqn-U_k^*=(n)X_{n,n}}), and (\ref{defn-U_{n,k}}) that 
\begin{equation}
\vc{U}_n^*
=
\left\{
\begin{array}{l@{~~~}l}
(-\vc{Q}_{0,0})^{-1}, & n=0,
\\
\left(
-\vc{Q}_{n,n} 
- 
\dm\sum_{\ell=0}^{n-1} \vc{U}_{n,\ell} \vc{Q}_{\ell,n}
\right)^{-1},
& n \in \bbN.
\end{array}
\right.
\label{defn-U_k^*}
\end{equation}
Using $\vc{U}_n^*$ and $\vc{U}_{n,k}$, we can express $\presub{(n)}\vc{X}_{n,k}$, $k\in\bbZ_n$, as follows.
\begin{lem}\label{lem-(n_s)X_{s,l}}
For $n \in \bbZ_+$,
\begin{equation}
\presub{(n)}\vc{X}_{n,k} 
= \vc{U}_n^* \vc{U}_{n,k},
\qquad k \in \bbZ_n,
\label{eqn-(n)X_{n,k}}
\end{equation}
and
\begin{equation}
\vc{U}_{n,k}
=
\left\{
\begin{array}{l@{~~}l}
(\vc{Q}_{n,n-1}   \vc{U}_{n-1}^*) 
(\vc{Q}_{n-1,n-2} \vc{U}_{n-2}^*)  \cdots 
(\vc{Q}_{k+1,k}\vc{U}_{k}^*), & k \in \bbZ_{n-1},
\\
\vc{I}, & k = n.
\end{array}
\right.
\label{defn-U_{k,l}}
\end{equation}
\end{lem}

\begin{proof}
Combining (\ref{defn-(n)X_{n,k}}) with (\ref{eqn-U_k^*=(n)X_{n,n}}) and (\ref{defn-U_{n,k}}), we have (\ref{eqn-(n)X_{n,k}}). Furthermore, applying (\ref{eqn-(n)X_{n,k}}) to (\ref{defn-U_{n,k}}) yields
\[
\vc{U}_{n,k}
= \vc{Q}_{n,n-1} \vc{U}_{n-1}^* \vc{U}_{n-1,k},\qquad k \in \bbZ_{n-1},
\]
which leads to (\ref{defn-U_{k,l}}). 
\qed
\end{proof}

\begin{rem}
A result similar to Lemma~\ref{lem-(n_s)X_{s,l}} is presented in 
Shin~\cite{Shin09} under the condition that $\presub{(n)}\vc{Q}$ is block tridiagonal (see Theorem~2.1 therein).
\end{rem}

\begin{rem}\label{rem-(n)X_{k,l}}
The matrices $\presub{(n)}\vc{X}_{k,\ell}$, $\ell \in \bbZ_k$, and $\vc{U}_{n,\ell}$, $\ell \in \bbZ_{n-1}$, have probabilistic interpretations.
The $(i,j)$-th element of $\presub{(n)}\vc{X}_{k,\ell}$ represents the expected total sojourn time in state $(\ell,j)$ before the first visit to $\ol{\bbS}_n$ (i.e., to any state above level $n$) starting from state $(k,i)$ (see, e.g., \cite[Theorem~2.4.3]{Lato99}).
Furthermore, the $(i,j)$-th element of $\vc{U}_{n,\ell}$ represents the expected total sojourn time  in state $(\ell,j)$ before the first visit to $\ol{\bbS}_n$ starting from state $(n,i)$, measured per unit of time spent in state $(n,i)$. Thus, we have (see \cite[Equation~(5.33)]{Lato99})
\begin{equation}
\vc{\pi}_{\ell} = \vc{\pi}_n \vc{U}_{n,\ell},
\qquad n \in \bbN,\ \ell \in \bbZ_n.
\label{eqn-pi_l-pi_k*U_{k,l}}
\end{equation}
\end{rem}

\medskip

We now obtain a matrix-product form of $\presub{(n)}\wh{\vc{\pi}}_{k}$, $k\in\bbZ_n$, by substituting (\ref{eqn-(n)X_{n,k}}) into (\ref{eqn-(n)pi_{n,k}}).
\begin{lem}\label{lem-product-form-(s)ol{pi}_s}
\begin{eqnarray}
\presub{(n)}\wh{\vc{\pi}}_{k}
&=&
{
\vc{\alpha}_n \vc{U}_n^* \vc{U}_{n,k}
\over 
\vc{\alpha}_n \sum_{\ell=0}^n \vc{U}_n^* \vc{U}_{n,\ell}\vc{e}
}, \qquad n \in \bbZ_+,\ k \in \bbZ_n. 
\label{eqn-(n_s)pi_{n_s,l}}
\end{eqnarray}
\end{lem}

\begin{rem}
Equations (\ref{ineqn-(n)X_{n,l}e>0}) and (\ref{eqn-(n)X_{n,k}}) lead to 
\begin{equation}
\sum_{\ell=0}^n \vc{U}_n^* \vc{U}_{n,\ell}\vc{e}>\vc{0},
\qquad n \in \bbZ_+,\ \ell \in \bbZ_n.
\label{ineqn-U_n^*-U_{n,l}e>0}
\end{equation}
\end{rem}

\subsection{An error bound}\label{subsec-error-bound}

In this subsection, we present an error bound for the LBCL-augmented truncation approximation $\presub{(n)}\wh{\vc{\pi}}$ to $\vc{\pi}$. The error bound is used to develop an algorithm for computing $\vc{\pi}$ in the next section.

To derive the error bound, we assume a Foster-Lyapunov drift condition.
\begin{cond}\label{cond-01}
The generator $\vc{Q}$ is irreducible, and there exist a constant $b \in (0,\infty)$, a finite set $\bbC \subset \bbS$, and a positive column vector $\vc{v}:=(v(k,i))_{(k,i)\in\bbS}$ such that $\inf_{(k,i)\in\bbS}v(k,i) > 0$ and
\begin{equation}
\vc{Q}\vc{v} \le  - \vc{e} + b \vc{1}_{\bbC},
\label{ineqn-QV<=-f+b1_C}
\end{equation}
where
$\vc{1}_{\bbB}:=(1_{\bbB}(k,i))_{(k,i)\in\bbS}$, $\bbB \subseteq \bbS$, denotes a column
vector defined by
\[
1_{\bbB}(k,i)
=\left\{
\begin{array}{l@{~~~}l}
1, & (k,i) \in \bbS,
\\
0, & (k,i) \in \bbS\setminus\bbB.
\end{array}
\right.
\]
\end{cond}

\begin{rem}
Recall that $\vc{Q}$ is the generator of the regular-jump Markov chain $\{(X(t),J(t))\}$ (see Section~\ref{sec-intro}) and thus $\vc{Q}$ is stable, i.e., $|q(\ell,j;\ell,j)| < \infty$ for all $(\ell,j) \in \bbS$ (see, e.g., \cite[Chapter~8, Definition~2.4 and Theorem~3.4]{Brem99}). The irreducibility of $\vc{Q}$ and the finiteness of $\bbC$ imply that
\[
\inf_{(k,i)\in\bbC} p^t(k,i;\ell,j) > 0\quad\mbox{for all $t > 0$ and $(\ell,j)\in \bbS$},
\]
which shows that $\bbC$ is a {\it small set} (see, e.g., \cite{Kont16}). Therefore, if Condition~\ref{cond-01} holds, then the irreducible generator $\vc{Q}$ is ergodic (see, e.g., \cite[Theorem~1.1]{Kont16}).
\end{rem}

Let $\vc{\Phi}^{(\beta)}:=(\phi^{(\beta)}(k,i;\ell,j)_{(k,i;\ell,j)\in\bbS}$ denote a stochastic matrix such that
\begin{equation*}
\vc{\Phi}^{(\beta)}
= \int_0^{\infty} \beta \re^{- \beta t} \vc{P}^{(t)} \rd t,
\qquad \beta > 0,
\end{equation*}
where $\vc{P}^{(t)}:=(p^{(t)}(k,i;\ell,j))_{(k,i;\ell,j)\in\bbS}$, $t \in \bbR_+$, is the transition matrix function of the Markov chain $\{(X(t),J(t))\}$ with generator $\vc{Q}$, i.e., 
\[
\PP((X(t),J(t)) = (\ell,j) \mid (X(0),J(0)) = (k,i)),\qquad
(k,i;\ell,j)\in\bbS^2.
\]
Because $\vc{Q}$ is ergodic, we have $\vc{\Phi}^{(\beta)} > \vc{O}$. We also define $\overline{\phi}_{\bbC}^{(\beta)}$, $\beta>0$, as
\begin{eqnarray*}
\overline{\phi}_{\bbC}^{(\beta)}
&=& \sup_{(\ell,j) \in \bbS}\min_{(k,i)\in\bbC} \phi^{(\beta)}(k,i;\ell,j) 
> 0,\qquad \beta > 0.
\end{eqnarray*}
We then have the following result from \cite[Theorem~2.1]{LLM2018} with $\vc{f}=\vc{g}=\vc{e}$.
\begin{prop}\label{prop-Lem-2.1-LLM2018}
Under Condition~\ref{cond-01}, the following holds:
\begin{equation}
\| \presub{(n)}\wh{\vc{\pi}} - \vc{\pi} \|
\le 2 \cdot \presub{(n)}\wh{\vc{\pi}}\hms
| \presub{(n)}\wh{\vc{Q}} - \vc{Q} |
\left(
\vc{v} 
+ 
{ b \over \beta\overline{\phi}_{\bbC}^{(\beta)} } \vc{e}
\right),\quad n\in\bbZ_+,~\beta > 0,
\label{bound-|D|g}
\end{equation}
where, for any vector $\vc{m}:=(m(i))$, $\|\vc{m}\|$ denotes the total variation norm of $\vc{m}$, i.e., $\|\vc{m}\| = \sum_i |m(i)|$.
\end{prop}

From Proposition~\ref{prop-Lem-2.1-LLM2018}, we derive a more informative bound for $\| \presub{(n)}\wh{\vc{\pi}} - \vc{\pi} \|$. For this purpose, we define 
 $\vc{U}^*_{n,k}$, $n \in \bbZ_+$, $k\in\bbZ_n$, as
\begin{equation}
\vc{U}^*_{n,k}
= \vc{U}_n^* \vc{U}_{n,k},
\qquad n \in \bbZ_+,\ k\in\bbZ_n.
\label{defn-U_{n,k}^*}
\end{equation}
We also define $\vc{u}^*_n:=(u_n^*(i))_{i\in\bbM_n}$, $n \in \bbZ_+$, as
\begin{equation}
\vc{u}^*_n
= \sum_{\ell=0}^n \vc{U}^*_{n,\ell}\vc{e}
=   \sum_{\ell=0}^n \vc{U}_n^* \vc{U}_{n,\ell}\vc{e} > \vc{0},
\qquad n \in \bbZ_+,
\label{defn-u_n^*}
\end{equation}
where $\vc{u}^*_n > \vc{0}$ due to (\ref{ineqn-U_n^*-U_{n,l}e>0}).
Using (\ref{defn-U_{n,k}^*}) and (\ref{defn-u_n^*}), we rewrite (\ref{eqn-(n_s)pi_{n_s,l}}) as
\begin{eqnarray}
\presub{(n)}\wh{\vc{\pi}}_k
= {
\vc{\alpha}_n \vc{U}^*_{n,k}
\over 
\vc{\alpha}_n \vc{u}^*_n
},
\qquad n \in \bbZ_+,\ k\in\bbZ_n.
\label{eqn-(n)wh{pi}_k}
\end{eqnarray}
\begin{thm}\label{thm-bound}
If Condition~\ref{cond-01} holds, then
\begin{eqnarray}
\| \presub{(n)}\wh{\vc{\pi}} - \vc{\pi} \|
&\le& E(n), \qquad n \in \bbZ_+,
\label{bound-pi-g}
\end{eqnarray}
where $E(\,\cdot\,):=E^{(\beta)}(\,\cdot\,)$, called the {\it error bound function}, is given by
\begin{eqnarray}
E(n)
&=& {2 \over \vc{\alpha}_n \vc{u}_n^* }
\!
\left\{
\vc{\alpha}_n \!
\left(
\vc{v}_n 
+ 
\sum_{k=0}^n \vc{U}_{n,k}^*
\sum_{\ell=n+1}^{\infty} \vc{Q}_{k,\ell} \vc{v}_{\ell}
\right)
\!
+ { 2b \over \beta\overline{\phi}_{\bbC}^{(\beta)} } 
\right\},~ n \in \bbZ_+, \quad~~~~
\label{defn-E(n)}
\end{eqnarray}
with $\beta > 0$.
\end{thm}

\begin{rem}
The error bound function $E$ has a free parameter $\beta>0$ involved in the intractable factor $\ol{\phi}_{\bbC}^{(\beta)}$. Thus, it is, in general, difficult to discuss theoretically how $\beta$ impacts on the decay speed of $E$. Through numerical experiments, Masuyama~\cite{Masu17-JORSJ} investigates such a problem for the {\it last-column block-augmented truncation}, though the function $E$ is referred to therein as the {\it error decay function}, instead of the error bound function. Note that the last-column block-augmented truncation belongs to the class of {\it block-augmented truncations} (see \cite{LiHai00} for details). Therefore, the last-column block-augmented truncation is indeed different from our LBCL-augmented truncation, though they are fairly similar. 
\end{rem}

\medskip

\noindent
{\it Proof of Theorem~\ref{thm-bound}~} 
Suppose that Condition~\ref{cond-01} holds. It then follows from (\ref{eqn-(n)wh{pi}((n)wh{Q}-Q)}) that
\begin{eqnarray}
\lefteqn{
\presub{(n)}\wh{\vc{\pi}}
|\presub{(n)}\wh{\vc{Q}} - \vc{Q}|
\left(
\vc{v} 
+ { b \over \beta\overline{\phi}_{\bbC}^{(\beta)} } \vc{e}
\right)
}
\quad &&
\nonumber
\\
&\le& 
\left(
\begin{array}{c@{~}|@{~}c}
\presub{(n)}\wh{\vc{\pi}} & \vc{0}
\end{array}
\right)
\left(
\begin{array}{c@{~}|@{~}c}
-\presub{(n)}\vc{Q} \vc{e} \presub{(n)}\wh{\vc{\alpha}} &
\presub{(n)}\vc{Q}_{>n}
\rule[-2.5mm]{0mm}{2mm}{}
\\
\hline
* & *
\end{array}
\right) 
\left(
\vc{v} 
+ { b \over \beta\overline{\phi}_{\bbC}^{(\beta)} } \vc{e}
\right)
\nonumber
\\
&=& \left(
\begin{array}{c@{~}|@{~}c}
\presub{(n)}\wh{\vc{\pi}} (-\presub{(n)}\vc{Q}\vc{e}) \cdot
\presub{(n)}\wh{\vc{\alpha}} &
\presub{(n)}\wh{\vc{\pi}} \presub{(n)}\vc{Q}_{>n}
\end{array}
\right) 
\left(
\vc{v} 
+ { b \over \beta\overline{\phi}_{\bbC}^{(\beta)} } \vc{e}
\right),
\quad n \in \bbZ_+. \qquad~
\label{eqn-diff-pi}
\end{eqnarray}
Substituting (\ref{eqn-diff-pi}) into (\ref{bound-|D|g}), we have, for $n \in \bbZ_+$,
\begin{eqnarray}
\| \presub{(n)}\wh{\vc{\pi}} -  \vc{\pi} \|
&\le&  2 
\left(
\begin{array}{c@{~}|@{~}c}
\presub{(n)}\wh{\vc{\pi}} (-\presub{(n)}\vc{Q}\vc{e}) \cdot
\presub{(n)}\wh{\vc{\alpha}} &
\vc{0}
\end{array}
\right) 
\left(
\vc{v} 
+ { b \over \beta\overline{\phi}_{\bbC}^{(\beta)} } \vc{e}
\right)
\nonumber
\\
&&  {} + 2
\left(
\begin{array}{c@{~}|@{~}c}
\vc{0} &
\presub{(n)}\wh{\vc{\pi}} \presub{(n)}\vc{Q}_{>n}
\end{array}
\right) 
\left(
\vc{v} 
+ { b \over \beta\overline{\phi}_{\bbC}^{(\beta)} } \vc{e}
\right)
\nonumber
\\
&=& 2 \presub{(n)}\wh{\vc{\pi}} 
(-\presub{(n)}\vc{Q} \vc{e}) 
\cdot
\left(
\vc{\alpha}_n \vc{v}_n 
+ { b \over \beta\overline{\phi}_{\bbC}^{(\beta)} }
\right)
\nonumber
\\
&& {}
+ 2
\left[
\sum_{k=0}^n \presub{(n)}\wh{\vc{\pi}}_k \sum_{\ell=n+1}^{\infty}
\vc{Q}_{k,\ell}
\left(
\vc{v}_{\ell} 
+ { b \over \beta\overline{\phi}_{\bbC}^{(\beta)} }\vc{e}
\right)
\right], \qquad~~~
\label{eqn-diff-pi-02}
\end{eqnarray}
where the last equality holds due to (\ref{eqn-(n)wh{alpha}v}) and $\presub{(n)}\wh{\vc{\alpha}}\vc{e} =1$ for $n \in \bbZ_+$. 
Because $\vc{Q}\vc{e} = \vc{0}$,
\[
\sum_{k=0}^n \presub{(n)}\wh{\vc{\pi}}_k \sum_{\ell=n+1}^{\infty}
\vc{Q}_{k,\ell}\vc{e}
= \sum_{k=0}^n \presub{(n)}\wh{\vc{\pi}}_k 
\left(-\sum_{\ell=0}^n \vc{Q}_{k,\ell}\vc{e} \right)
= \presub{(n)}\wh{\vc{\pi}} (-\presub{(n)}\vc{Q}\vc{e}).
\]
Incorporating this into (\ref{eqn-diff-pi-02}), we have
\begin{eqnarray}
\| \presub{(n)}\wh{\vc{\pi}} -  \vc{\pi} \|
&\le& 2 \presub{(n)}\wh{\vc{\pi}} 
(-\presub{(n)}\vc{Q} \vc{e}) 
\cdot
\left(
\vc{\alpha}_n \vc{v}_n 
+ { 2b \over \beta\overline{\phi}_{\bbC}^{(\beta)} }
\right)
\nonumber
\\
&& {}
+ 2
\left[
\sum_{k=0}^n \presub{(n)}\wh{\vc{\pi}}_k \sum_{\ell=n+1}^{\infty}
\vc{Q}_{k,\ell}
\vc{v}_{\ell}
\right].
\label{eqn-diff-pi-03}
\end{eqnarray}
Note here that (\ref{defn-(n)ol{pi}_n}) and $\presub{(n)}\wh{\vc{\alpha}}\vc{e}=1$ yield 
\[
 \presub{(n)}\wh{\vc{\pi}} (-\presub{(n)}\vc{Q}\vc{e})
= 
{ 
\presub{(n)}\wh{\vc{\alpha}}\vc{e}
\over 
\presub{(n)}\wh{\vc{\alpha}} (-\presub{(n)}\vc{Q})^{-1}\vc{e}
}
= 
{
1
\over 
\presub{(n)}\wh{\vc{\alpha}} (-\presub{(n)}\vc{Q})^{-1}\vc{e}
}.
\]
Thus, we can rewrite (\ref{eqn-diff-pi-03}) as
\begin{eqnarray}
\| \presub{(n)}\wh{\vc{\pi}} -  \vc{\pi} \|
&\le&  
{
2
\over 
\presub{(n)}\wh{\vc{\alpha}} (-\presub{(n)}\vc{Q})^{-1}\vc{e}
}
\left(
\vc{\alpha}_n \vc{v}_n 
+ { 2b \over \beta\overline{\phi}_{\bbC}^{(\beta)} }
\right)
\nonumber
\\
&& {}
+ 2
\sum_{k=0}^n \presub{(n)}\wh{\vc{\pi}}_k \sum_{\ell=n+1}^{\infty}
\vc{Q}_{k,\ell}
\vc{v}_{\ell}. 
\label{eqn-diff-pi-04}
\end{eqnarray}
Furthermore, from (\ref{eqn-(n)wh{alpha}(-(n)Q)^{-1}}), 
(\ref{eqn-(n)X_{n,k}}), and (\ref{defn-u_n^*}), we have 
\begin{equation*}
\presub{(n)}\wh{\vc{\alpha}}(-\presub{(n)}\vc{Q})^{-1}\vc{e}
= \vc{\alpha}_n \sum_{\ell=0}^n\vc{U}_n^*\vc{U}_{n,\ell}\vc{e}
= \vc{\alpha}_n \vc{u}_n^*,
\qquad n \in \bbZ_+.
\end{equation*}
Applying this equation and (\ref{eqn-(n)wh{pi}_k}) to (\ref{eqn-diff-pi-04}), we obtain
\begin{eqnarray*}
\| \presub{(n)}\wh{\vc{\pi}} -  \vc{\pi} \|
&\le&  
{
2
\over 
\vc{\alpha}_n \vc{u}_n^*
}
\left(
\vc{\alpha}_n \vc{v}_n 
+ { 2b \over \beta\overline{\phi}_{\bbC}^{(\beta)} }
\right)
\nonumber
\\
&& {}
+ 2
\sum_{k=0}^n  
{\vc{\alpha}_n \vc{U}_{n,k}^* \over \vc{\alpha}_n \vc{u}_n^*} 
\sum_{\ell=n+1}^{\infty} \vc{Q}_{k,\ell} \vc{v}_{\ell}
\nonumber
\\
&=&  
{ 2
\over
\vc{\alpha}_n \vc{u}_n^*
}
\left\{
 \vc{\alpha}_n
\left(
\vc{v}_n 
+ 
\sum_{k=0}^n \vc{U}_{n,k}^*
\sum_{\ell=n+1}^{\infty} \vc{Q}_{k,\ell} \vc{v}_{\ell}
\right)
+ { 2b \over \beta\overline{\phi}_{\bbC}^{(\beta)} } 
\right\},
\end{eqnarray*}
which results in (\ref{bound-pi-g}) together with (\ref{defn-E(n)}).  \qed

\subsection{A counterexample to convergence}\label{subsec-example}

In the previous subsection, we have established the error bound for the LBCL-augmented truncation approximation $\presub{(n)}\wh{\vc{\pi}}$. 
We note that, even if the truncation parameter $n$ goes to infinity, $\presub{(n)}\wh{\vc{\pi}}$ does not necessarily converge to $\vc{\pi}$, in general.
However, it always holds that $\lim_{n\to\infty}\presub{(n)}\wh{\vc{\pi}}=\vc{\pi}$ for {\it special} upper block-Hessenberg Markov chains such that the block matrices $\vc{Q}_{k,\ell}$ are scalars. For such a special case, Gibson and Seneta \cite{Gibs87-JAP} prove that any augmented truncation approximation converges to the original stationary distribution as the truncation parameter goes to infinity (see Theorem 2.2 therein). Of course, this is not the case for {\it general} upper block-Hessenberg Markov chains. Indeed, we introduce a counterexample \cite{Kimura-Takine16}.

Fix $\bbM_n = \{1,2\}$ for all $n \in \bbZ_+$, and assume that the block matrices $\vc{Q}_{k,\ell}$ satisfy the following:
\begin{align}
\vc{Q}_{n,n} 
&=
\left(
\begin{array}{cc}
\star & \star
\\
0 & \star
\end{array}
\right),
&
\vc{Q}_{n,n+1} 
&=
\left(
\begin{array}{cc}
\star & 0
\\
\star & \star
\end{array}
\right),
&
n &\in\bbZ_+,
\label{special-Q_{n,n}}
\\
\vc{Q}_{2k-1,2k-2} 
&=
\left(
\begin{array}{cc}
\star & 0
\\
0 & \star
\end{array}
\right),
&
\vc{Q}_{2k,2k-1} 
&=
\left(
\begin{array}{cc}
\star & 0
\\
0 & 0
\end{array}
\right), & k &\in \bbN.
\label{special-Q_{2k-1,2k-2}}
\end{align}
where the symbol ``\, $\star$\, " denotes some nonzero element. In this case, $\vc{Q}$ is irreducible (see Figure~\ref{fig-transition}), but $\bbS_{2k-1}$ is not reachable from state $(2k,2)$ avoiding $\ol{\bbS}_{2k}$. 
\begin{figure*}[htb]
\centering
\includegraphics[scale=0.4]{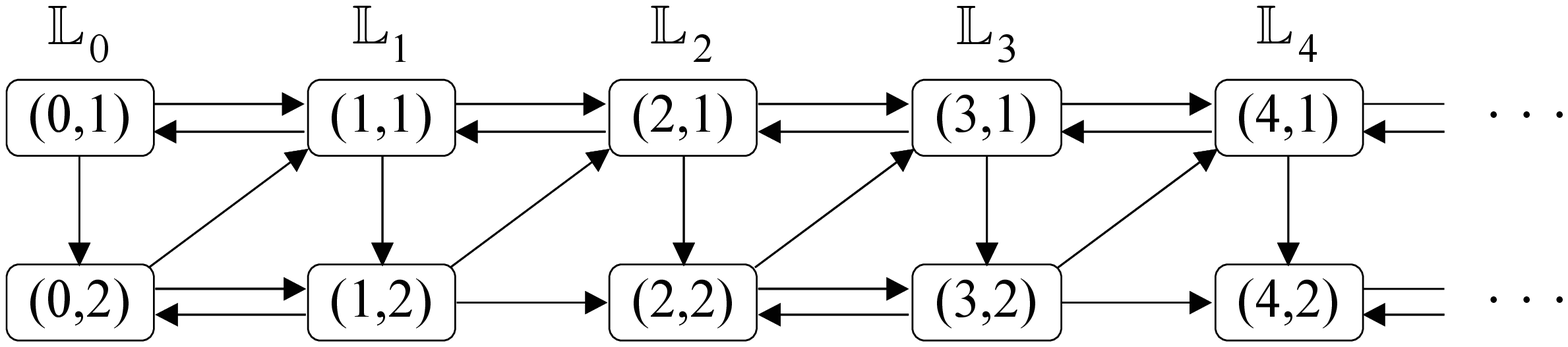}
\caption{Transition diagram}
\label{fig-transition}
\end{figure*}
Thus, the probabilistic interpretations (see Remark~\ref{rem-(n)X_{k,l}}) of the matrices $\vc{U}_{2k}^*=\presub{(2k)}\vc{X}_{2k,2k}$ and $\vc{U}_{2k,\ell}$, $\ell\in\bbZ_{2k-1}$, implies that
\begin{align}
&&&&
\vc{U}_{2k}^*
&= \left(
\begin{array}{cc}
\star & \star
\\
0 & \star
\end{array}
\right),& k &\in \bbN,&&&&
\label{eqn-U_{2k}^*}
\\
&&&&
\vc{U}_{2k,\ell}
&= \left(
\begin{array}{cc}
\star & \star
\\
0 & 0
\end{array}
\right), & k &\in \bbN,~\ell\in\bbZ_{2k-1}.&&&&
\label{eqn-U_{2k,l}}
\end{align}

We now assume that  $\vc{Q}$ is ergodic. We then set
\begin{equation}
\vc{\alpha}_n = (0,1),\qquad n \in \bbZ_+,
\label{special-alpha_n}
\end{equation}
which implies that $\presub{(n)}\wh{\vc{\pi}}$ is the last-column-augmented truncation approximation to the stationary distribution vector $\vc{\pi} > \vc{0}$ of $\vc{Q}$. Applying (\ref{eqn-U_{2k}^*}), (\ref{eqn-U_{2k,l}}), and (\ref{special-alpha_n}) to (\ref{eqn-(n_s)pi_{n_s,l}}) yields
\[
\presub{(2k)}\wh{\vc{\pi}} = (0,\dots,0,1),\qquad k \in \bbN,
\]
which shows that $\{\presub{(n)}\wh{\vc{\pi}};n\in\bbZ_+\}$ does not converge to $\vc{\pi}$ in the present setting.

The example presented here implies that, in some cases, the convergence of $\{\presub{(n)}\wh{\vc{\pi}}\}$ to $\vc{\pi}$ can require an adaptive choice of 
 the augmentation distribution vector $\presub{(n)}\vc{\alpha}$, depending on $n$. We discuss this problem in the next section.


\section{Main results}\label{sec-algorithm}

This section is divided into three subsections. In Section~\ref{subsec-LFP}, we formulate linear fractional programming (LFP) problems for finding augmentation distribution vectors such that the error bound function $E$ converges to zero, i.e., $\lim_{n\to\infty}E(n) = 0$. In Section~\ref{subsec-MIP}, using the optimal solutions of these LFP problems, we construct an MIP form of $\vc{\pi}$. In Section~\ref{subsec-computation}, we present a sequential update algorithm for computing the MIP form.

In this section, we assume that Condition~\ref{cond-01} holds, as in Section~\ref{subsec-error-bound}. We also assume that $n$ takes an arbitrary value in $\bbZ_+$, unless otherwise stated.

\subsection{LFP problems for an MIP form of the stationary distribution vector}\label{subsec-LFP}

Consider the following LFP problem for each $n \in \bbZ_+$:

\smallskip

\begin{subequations}\label{prob-01}
\begin{align}
&&&&
&\mbox{Minimize}     	
&  	
& r_n(\vc{\alpha}_n):=
{\vc{\alpha}_n\vc{y}_n  \over \vc{\alpha}_n\vc{u}_n^*};
\label{defn-r_n}
&&&&
\\
&&&&
&\mbox{Subject to}		
&	
& \vc{\alpha}_n \ge \vc{0},
\label{const-x>=0}
&&&&
\\
&&&& 
&		
&	
& \vc{\alpha}_n\vc{e} = 1,
&&&& \label{const-xe=1}
\end{align}
where $\vc{y}_n:=(y_n(i))_{i\in\bbM_n}$ denotes
\begin{equation}
\vc{y}_n = 
\vc{v}_n 
+ 
\sum_{k=0}^n \vc{U}_{n,k}^*
\sum_{\ell=n+1}^{\infty} \vc{Q}_{k,\ell} \vc{v}_{\ell} > \vc{0}.
\label{defn-y_n}
\end{equation}
\end{subequations}
It follows from (\ref{defn-E(n)}), (\ref{defn-r_n}), and (\ref{defn-y_n}) that
\begin{equation}
E(n)
= 2\left( 
r_n(\vc{\alpha}_n) 
+ { 1 \over \vc{\alpha}_n \vc{u}_n^{\ast} }
{ 2b \over \beta\overline{\phi}_{\bbC}^{(\beta)} } 
\right), \qquad n \in \bbZ_+.
\label{bound-(n)pi-pi}
\end{equation}
Furthermore, let $\vc{\alpha}_n^*:=(\alpha_n^*(j))_{j\in\bbM_n}$ denote a probability vector such that
\begin{equation}
\alpha_n^*(j)
=
\left\{
\begin{array}{l@{~~~}l}
1, & j=j_n^*,
\\
0, & j \neq j_n^*,
\end{array}
\right.
\label{defn-psi_n^*}
\end{equation}
where
\begin{equation}
j_n^* 
\in \argmin_{j\in\bbM_n} {y_n(j)  \over u_n^*(j)}.
\label{defn-j_n^*}
\end{equation}
We then have the following theorem. 
\begin{thm}\label{thm-optima-solution-LFP}
For each $n \in \bbZ_+$, the probability vector $\vc{\alpha}_n^*$ is an optimal solution of the LFP problem~(\ref{prob-01}).
\end{thm}

\begin{proof}
From (\ref{defn-psi_n^*}) and (\ref{defn-j_n^*}), we have
\[
\xi_n := {\vc{\alpha}_n^* \vc{y}_n \over \vc{\alpha}_n^* \vc{u}_n^*}
= {y_n(j_n^*) \over u_n^*(j_n^*)}
= \min_{j\in\bbM_n}{y_n(j) \over u_n^*(j)} > 0,
\]
which leads to $\vc{y}_n \ge \xi_n \vc{u}_n^* > \vc{0}$. Thus, for any $1 \times M_n$ probability vector $\vc{p}_n$, we obtain
\[
{\vc{p}_n\vc{y}_n \over \vc{p}_n\vc{u}_n^*} \ge \xi_n
= {\vc{\alpha}_n^* \vc{y}_n \over \vc{\alpha}_n^* \vc{u}_n^*}.
\]
Therefore, $\vc{\alpha}_n^*$ is an optimal solution of the LFP problem~(\ref{prob-01}).
\qed
\end{proof}

\subsection{An MIP form of the stationary distribution vector}\label{subsec-MIP}

Let $\presub{(n)}\wh{\vc{\pi}}^*
:= (
\presub{(n)}\wh{\vc{\pi}}_{0}^*,
\presub{(n)}\wh{\vc{\pi}}_{1}^*,
\dots,
\presub{(n)}\wh{\vc{\pi}}_{n}^*)$ denote a probability vector such that
\begin{equation}
\presub{(n)}\wh{\vc{\pi}}_{k}^*
= {
\vc{\alpha}_n^*\vc{U}_{n,k}^*
\over
\vc{\alpha}_n^* \vc{u}_n^*
}
= {
{\rm row}\{\vc{U}_{n,k}^* \}_{j_n^*} 
\over  
u_n^*(j_n^*) 
},\qquad k \in \bbZ_n,
\label{defn-(n)wh{pi}_k^*}
\end{equation}
where ${\rm row}\{\,\cdot\,\}_j$ denotes the $j$-th row of the matrix in the brackets. Note here that $\presub{(n)}\wh{\vc{\pi}}_{k}^*$ is equal to $\presub{(n)}\wh{\vc{\pi}}_{k}$ in (\ref{eqn-(n)wh{pi}_k}) with $\vc{\alpha}_n=\vc{\alpha}_n^*$. Therefore, it follows from Theorem~\ref{thm-bound} that
\begin{equation}
\| \presub{(n)}\wh{\vc{\pi}}^* - \vc{\pi} \| 
\le E^*(n),\qquad n \in \bbZ_+,
\label{bound-(n)wh{pi}^*-02}
\end{equation}
where function $E^*$ is equal to $E$ given in (\ref{bound-(n)pi-pi}) with $\vc{\alpha}_n=\vc{\alpha}_n^*$; that is,
\begin{eqnarray}
E^*(n)
&=&
2\left(
r_n( \vc{\alpha}_n^*)
+ 
{ 1 \over \vc{\alpha}_n^* \vc{u}_n^{\ast} }
{ 2b \over \beta\overline{\phi}_{\bbC}^{(\beta)} }
\right), \qquad n \in \bbZ_+.
\label{defn-E^*(n)}
\end{eqnarray}

To proceed further, we assume the following.
\begin{cond}\label{cond-02}
\begin{eqnarray}
\sum_{n=0}^{\infty} \vc{\pi}_n \vc{\Delta}_n \vc{v}_n < \infty,
\label{lim-pi_n-Delta_n-v_n}
\end{eqnarray}
where $\vc{\Delta}_n:=(\Delta_n(i,j))_{i,j\in\bbM_n}$ denotes an $M_n \times M_n$ diagonal matrix such that 
\begin{equation}
\Delta_n(i,i) = |q(n,i;n,i)|,\qquad i \in \bbM_n.
\label{defn-Delta_n}
\end{equation}
\end{cond}

\begin{lem}\label{lem-r^{(1)}}
Suppose that Conditions~\ref{cond-01} and \ref{cond-02} hold. We then have
\begin{equation}
\lim_{n\to\infty} r_n(\vc{\alpha}_n^*) 
= \lim_{n\to\infty} {\vc{\alpha}_n^*\vc{y}_n \over \vc{\alpha}_n^*\vc{u}_n^*}
= 0.
\label{lim-r_n(psi_n^*)}
\end{equation}
\end{lem}

\begin{rem}\label{rem-Q-exp-bounded}
If $\vc{Q}$ is bounded, i.e., $\sup_{(n,i)\in\bbS}\Delta_n(i,i) < \infty$,
then Condition~\ref{cond-02} is reduced to
\begin{equation*}
\vc{\pi} \vc{v} = \sum_{n=0}^{\infty} \vc{\pi}_n \vc{v}_n < \infty.
\end{equation*}
\end{rem}

\noindent
{\it Proof of Lemma~\ref{lem-r^{(1)}}~} 
To prove this lemma, we require the following proposition (which is proved in Appendix~\ref{appen-ineqn-pi-(U_n^*)^{-1}}).
\begin{prop}\label{lem-pi_n-Lambda_n}
Under Condition~\ref{cond-01},
\begin{equation}
\vc{\pi}_n \vc{\Delta}_n
\ge
\vc{\pi}_n (\vc{U}_n^*)^{-1} 
\ge \vc{0}, \neq \vc{0}
\quad\mbox{for all $n \in \bbZ_+$}.
\label{ineqn-pi-(U_n^*)^{-1}}
\end{equation}
\end{prop}

Let $\wt{\vc{\alpha}}_n$ denote 
\begin{equation}
\wt{\vc{\alpha}}_n
= {
\vc{\pi}_n (\vc{U}_n^*)^{-1} 
\over 
\vc{\pi}_n (\vc{U}_n^*)^{-1} \vc{e}}\ge \vc{0}, \neq \vc{0},
\label{fix-x-alpha_n}
\end{equation}
which is well-defined due to Proposition~\ref{lem-pi_n-Lambda_n}. Note that $\wt{\vc{\alpha}}_n$ is a feasible solution of the LFP problem~(\ref{prob-01}). Thus, by the optimality of $\vc{\alpha}_n^*$, we have
\begin{equation*}
r_n(\vc{\alpha}_n^*) 
\le r_n(\wt{\vc{\alpha}}_n)\quad \mbox{for all $n \in \bbZ_+$}.
\end{equation*}
It follows from (\ref{defn-r_n}) and (\ref{fix-x-alpha_n}) that
\begin{eqnarray}
r_n(\wt{\vc{\alpha}}_n)
&=& {
\vc{\pi}_n (\vc{U}_n^*)^{-1} \vc{y}_n 
\over 
\vc{\pi}_n (\vc{U}_n^*)^{-1} \vc{u}_n^*}
\nonumber\\
&=&
{1 \over \vc{\pi}_n (\vc{U}_n^*)^{-1} \vc{u}_n^*}
\nonumber
\\
&& \times 
\left(
\vc{\pi}_n (\vc{U}_n^*)^{-1}\vc{v}_n  + 
\vc{\pi}_n (\vc{U}_n^*)^{-1}
\sum_{k=0}^n \vc{U}_{n,k}^*
\sum_{\ell=n+1}^{\infty} \vc{Q}_{k,\ell} \vc{v}_{\ell}
\right),\qquad
\label{eqn-180326-01}
\end{eqnarray}
where the second equality holds due to (\ref{defn-y_n}). It also follows from (\ref{eqn-pi_l-pi_k*U_{k,l}}), (\ref{defn-U_{n,k}^*}), and (\ref{defn-u_n^*}) that
\begin{eqnarray*}
\vc{\pi}_n (\vc{U}_n^*)^{-1} \vc{U}_{n,k}^*
&=& \vc{\pi}_n \vc{U}_{n,k}
= \vc{\pi}_k,\qquad k \in \bbZ_n,
\\
\vc{\pi}_n (\vc{U}_n^*)^{-1} \vc{u}_n^*
&=&  \sum_{\ell=0}^n  
\left\{ 
\vc{\pi}_n (\vc{U}_n^*)^{-1} \vc{U}_{n,\ell}^*
\right\}\vc{e}
= \sum_{\ell=0}^n \vc{\pi}_{\ell}\vc{e}.
\end{eqnarray*}
Substituting these equations into (\ref{eqn-180326-01}), and using (\ref{ineqn-pi-(U_n^*)^{-1}}), we obtain 
\begin{eqnarray}
r_n(\wt{\vc{\alpha}}_n)
&=& 
{
1
\over 
\sum_{\ell=0}^n \vc{\pi}_{\ell}  \vc{e}
}
\left(
\vc{\pi}_n (\vc{U}_n^*)^{-1}\vc{v}_n 
+ 
\sum_{k=0}^n \vc{\pi}_k
\sum_{\ell=n+1}^{\infty} \vc{Q}_{k,\ell} \vc{v}_{\ell}
\right)
\nonumber
\\
&\le&
{
1
\over 
\sum_{\ell=0}^n \vc{\pi}_{\ell}  \vc{e}
}
\left(
\vc{\pi}_n \vc{\Delta}_n \vc{v}_n 
+ 
\sum_{k=0}^n \vc{\pi}_k
\sum_{\ell=n+1}^{\infty} \vc{Q}_{k,\ell} \vc{v}_{\ell}
\right).
\label{ineqn-r_n(psi_n^*)}
\end{eqnarray}
Consequently, the proof of (\ref{lim-r_n(psi_n^*)}) is completed by showing that
the right-hand side of (\ref{ineqn-r_n(psi_n^*)}) converges to zero as $n\to\infty$.

It follows from (\ref{ineqn-QV<=-f+b1_C}) that, for all $n \in \bbZ_+$ and $k \in \bbZ_n$,
\begin{eqnarray*}
\vc{0}
\le
\sum_{\ell=n+1}^{\infty} \vc{Q}_{k,\ell} \vc{v}_{\ell}
&\le& - \sum_{\ell=0}^n \vc{Q}_{k,\ell} \vc{v}_{\ell} - \vc{e} + b\vc{e}
\nonumber
\\
&\le&  - \vc{Q}_{k,k} \vc{v}_k + b\vc{e}
\le \vc{\Delta}_k \vc{v}_k + b\vc{e},
\end{eqnarray*}
and thus
\begin{equation*}
\sum_{k=0}^n \vc{\pi}_k
\sum_{\ell=n+1}^{\infty} \vc{Q}_{k,\ell} \vc{v}_{\ell}
\le \sum_{k=0}^{\infty} \vc{\pi}_k\vc{\Delta}_k \vc{v}_k + b< \infty
\quad \mbox{for all $n \in \bbZ_+$},
\end{equation*}
where the last inequality holds due to (\ref{lim-pi_n-Delta_n-v_n}). Therefore, by the dominated convergence theorem,
\begin{eqnarray}
\lim_{n\to\infty} 
\sum_{k=0}^n \vc{\pi}_k
\sum_{\ell=n+1}^{\infty} \vc{Q}_{k,\ell} \vc{v}_{\ell}
= 
\sum_{k=0}^{\infty} \vc{\pi}_k
\lim_{n\to\infty} 
\sum_{\ell=n+1}^{\infty} \vc{Q}_{k,\ell} \vc{v}_{\ell}
= \vc{0}.
\label{eqn-180318-02}
\end{eqnarray}
It also follows from (\ref{lim-pi_n-Delta_n-v_n}) that
\begin{equation}
\lim_{n\to\infty} \vc{\pi}_n \vc{\Delta}_n \vc{v}_n = 0.
\label{eqn-180318-03}
\end{equation}
Combining (\ref{eqn-180318-02}), (\ref{eqn-180318-03}), and $\sum_{\ell=0}^{\infty} \vc{\pi}_{\ell} \vc{e}=1$, we obtain
\[
\lim_{n\to\infty}
{1 \over 
\sum_{\ell=0}^n \vc{\pi}_{\ell}  \vc{e}
}
\left(
\vc{\pi}_n \vc{\Delta}_n \vc{v}_n 
+ 
\sum_{k=0}^n \vc{\pi}_k
\sum_{\ell=n+1}^{\infty} \vc{Q}_{k,\ell} \vc{v}_{\ell}
\right)
=0,
\]
which completes the proof. \qed

\medskip

The following theorem is a consequence of Lemma~\ref{lem-r^{(1)}} together with (\ref{bound-(n)wh{pi}^*-02}) and (\ref{defn-E^*(n)}).
\begin{thm}\label{thm-E^*-01}
Suppose that Conditions~\ref{cond-01} and \ref{cond-02} hold. We then have
\begin{equation}
\lim_{n\to\infty} E^*(n) = 0,
\label{convergence-(n)ol{pi}_n^*}
\end{equation}
and thus (\ref{bound-(n)wh{pi}^*-02}) yields
\begin{eqnarray}
\lim_{n\to\infty} \| \presub{(n)}\wh{\vc{\pi}}^* - \vc{\pi} \|  &=& 0.
\label{convergence-(n)wh{pi}^*}
\end{eqnarray}
\end{thm}

\begin{proof}
We prove only (\ref{convergence-(n)ol{pi}_n^*}).
It follows from (\ref{defn-y_n}), (\ref{defn-psi_n^*}), and $\inf_{(k,i)\in\bbS}v(k,i) > 0$ (see Condition~\ref{cond-01}) that
\[
\vc{\alpha}_n^* \vc{y}_n \ge \vc{\alpha}_n^* \vc{v}_n = v(n,j_n^*) > 0,
\qquad n \in \bbZ_+.
\]
Therefore, (\ref{lim-r_n(psi_n^*)}) implies that
\[
\lim_{n\to\infty} \vc{\alpha}_n^* \vc{u}_n^* = \infty,
\]
which yields 
\begin{equation}
\lim_{n\to\infty} 
{1 \over \vc{\alpha}_n^* \vc{u}_n^* }
{2b  \over \beta \ol{\phi}_{\bbC}^{(\beta)} }
= 0.
\label{eqn-180318-04}
\end{equation}
Applying (\ref{eqn-180318-04}) and Lemma~\ref{lem-r^{(1)}} to (\ref{defn-E^*(n)}) results in  (\ref{convergence-(n)ol{pi}_n^*}).
\qed
\end{proof}

\medskip

Theorem~\ref{thm-E^*-01} yields a {\it matrix-infinite-product (MIP) form} of $\vc{\pi}=(\vc{\pi}_0,\vc{\pi}_1,\dots)$ under Conditions~\ref{cond-01} and \ref{cond-02}. This is summarized in the following corollary.
\begin{coro}\label{coro-MIP-form-solution}
If Conditions~\ref{cond-01} and \ref{cond-02} hold, then
\begin{equation}
\vc{\pi}_k
= \lim_{n\to\infty}
{ 
\vc{\alpha}_n^* \vc{U}_{n,k}^*
\over 
\vc{\alpha}_n^* \vc{u}_n^*
},
\qquad k \in \bbZ_+,
\label{limit-form-pi_k}
\end{equation}
or equivalently,
\begin{equation}
\vc{\pi}_k
= \lim_{n\to\infty}
{ 
\vc{\alpha}_n^* \vc{U}_n^* \vc{U}_{n-1}\vc{U}_{n-2} \cdots \vc{U}_k
\over 
\vc{\alpha}_n^* 
\sum_{\ell=0}^n \vc{U}_n^* \vc{U}_{n-1}\vc{U}_{n-2} \cdots \vc{U}_{\ell}\vc{e}
},
\qquad k \in \bbZ_+,
\label{MIP-form-solution}
\end{equation}
where
\begin{equation}
\vc{U}_k = \vc{Q}_{k+1,k}\vc{U}_k^*,\qquad k \in \bbZ_+.
\label{defn-U_k}
\end{equation}
\end{coro}

\begin{proof}
Suppose that Conditions~\ref{cond-01} and \ref{cond-02} hold.
It then follows from (\ref{defn-(n)wh{pi}_k^*}) and (\ref{convergence-(n)wh{pi}^*}) that
\[
\vc{\pi}_k
= \lim_{n\to\infty} \presub{(n)}\wh{\vc{\pi}}_k^*
= \lim_{n\to\infty} {\vc{\alpha}_n^*\vc{U}_{n,k}^* \over \vc{\alpha}_n^*\vc{u}_n^*},\qquad k\in\bbZ_+,
\]
which shows that (\ref{limit-form-pi_k}) holds. Furthermore, combining (\ref{defn-U_{n,k}^*}) with (\ref{defn-U_{k,l}}) and (\ref{defn-U_k}) yields, for $n \in \bbZ_+$,
\[
\vc{U}_{n,k}^* = 
\left\{
\begin{array}{ll}
\vc{U}_n^*\vc{U}_{n-1}\vc{U}_{n-2} \cdots \vc{U}_k, & k\in\bbZ_{n-1},
\\
\vc{I},& k=n.
\end{array}
\right.
\]
Using this and (\ref{defn-u_n^*}), we can rewrite (\ref{limit-form-pi_k}) as (\ref{MIP-form-solution}).\qed
\end{proof}

\begin{rem}
Theorem~\ref{thm-E^*-01} ensures that the convergence in (\ref{limit-form-pi_k}) and (\ref{MIP-form-solution}) is uniform for $k \in \bbZ_+$.
\end{rem}

\begin{rem}
Another MIP form of $\vc{\pi}_k$ is presented in the preprint \cite{Masu16-arXiv:1603}, under some technical conditions different from Conditions~\ref{cond-01} and \ref{cond-02}.
\end{rem}

\subsection{A sequential update algorithm for the MIP form}\label{subsec-computation}

In this subsection, we propose an algorithm for computing $\vc{\pi}$, based on Theorem~\ref{thm-E^*-01} and Corollary~\ref{coro-MIP-form-solution}. Our algorithm sequentially updates the LBCL-augmented truncation approximation so that it converges to the MIP form (\ref{limit-form-pi_k}) of $\vc{\pi}$. 

To efficiently achieve this update procedure, we derive recursive formulas. 
Combining (\ref{defn-U_{n,k}^*}) with
(\ref{defn-U_k^*}) and (\ref{defn-U_{k,l}}), we have
\begin{subequations}\label{recursion-U_{n,k}^*}
\begin{eqnarray}
\vc{U}_{0,0}^*
&=& \vc{U}_0^* = (-\vc{Q}_{0,0})^{-1},
\label{eqn-U_{0,0}^*}
\\
\vc{U}_{n,k}^*
&=& 
\left\{
\begin{array}{l@{~~~}ll}
\vc{U}_n^* \vc{Q}_{n,n-1} \cdot \vc{U}^*_{n-1,k},
& n \in \bbN,\ & k \in \bbZ_{n-1},
\\
\vc{U}_n^*,
& n \in \bbN,\ & k = n.
\end{array}
\right.
\label{eqn-U_{n,k}^*}
\end{eqnarray}
\end{subequations}
Using (\ref{defn-u_n^*}), (\ref{eqn-U_{0,0}^*}), and
(\ref{eqn-U_{n,k}^*}), we also obtain
\begin{subequations}\label{recursion-u_n^*}
\begin{eqnarray}
\vc{u}^*_0
&=& \vc{U}_0^* \vc{e} = (-\vc{Q}_{0,0})^{-1}\vc{e},
\label{eqn-u_0^*}
\\
\vc{u}^*_n
&=& \vc{U}_n^* 
\left( \vc{e} + \vc{Q}_{n,n-1} \vc{u}^*_{n-1}
\right),
\qquad n \in \bbN.
\label{eqn-u_n^*}
\end{eqnarray}
\end{subequations}
Furthermore, (\ref{defn-U_{k,l}}) and (\ref{defn-U_{n,k}^*}) yield
\[
\vc{U}_{n,\ell} 
= \vc{Q}_{n,n-1}\vc{U}_{n-1}^*\vc{U}_{n-1,\ell}
= \vc{Q}_{n,n-1}\vc{U}_{n-1,\ell}^*,\qquad \ell \in \bbZ_{n-1}.
\]
Substituting this into (\ref{defn-U_k^*}) leads to
\begin{eqnarray}
\vc{U}_n^*
&=& 
\left(
- \vc{Q}_{n,n} 
- \vc{Q}_{n,n-1} \sum_{\ell=0}^{n-1} \vc{U}_{n-1,\ell}^*\vc{Q}_{\ell,n}
\right)^{-1},
\qquad n \in \bbN.
\label{eqn-U_n^*}
\end{eqnarray}

Our algorithm is composed of the equations (\ref{recursion-U_{n,k}^*})--(\ref{eqn-U_n^*}), Theorem~\ref{thm-E^*-01}, and Corollary~\ref{coro-MIP-form-solution}.
\begin{algorithm}[H]
\renewcommand{\thealgorithm}{}
\renewcommand{\labelenumi}{\arabic{enumi}.}
\renewcommand{\labelenumii}{(\alph{enumii})}
\renewcommand{\labelenumiii}{\roman{enumiii}.}
\caption{Computing the MIP form of $\vc{\pi}$}
{\bf Input}: $\vc{Q}$, $\varepsilon \in (0,1)$, and increasing sequence $\{n_{\ell};\ell\in \bbZ_+\}$ of positive integers.
\\ 
{\bf Output}: 
$\presub{(n)}\wh{\vc{\pi}}^* = (\presub{(n)}\wh{\vc{\pi}}_{0}^*,\presub{(n)}\wh{\vc{\pi}}_{1}^*,\dots,\presub{(n)}\wh{\vc{\pi}}_{n}^*)$,
where $n \in \bbZ_+$ is fixed when the iteration stops.
%
%
\begin{enumerate}
\setlength{\parskip}{0cm} 
\setlength{\itemsep}{0cm} 
\item Find $\vc{v}>\vc{0}$, $b>0$, and $\bbC \in \bbS$ such that Conditions~\ref{cond-01} and \ref{cond-02} hold.
\item Set $n = 0$ and $\ell=1$.
\item Compute $\vc{U}_0^*$ by (\ref{eqn-U_{0,0}^*}) and $\vc{u}_0^*$ by (\ref{eqn-u_0^*}).
\item Iterate (a)--(d) below:
%
%
\begin{enumerate}
\setlength{\parskip}{0cm} 
\setlength{\itemsep}{0cm} 
\item Increment $n$ by one.
\item Compute $\vc{U}_n^*=\vc{U}_{n,n}^*$ by (\ref{eqn-U_n^*}).
\item Compute $\vc{U}_{n,k}^*$, $k=0,1,\dots,n-1$, by
  (\ref{eqn-U_{n,k}^*}) and  $\vc{u}^*_n$ by
  (\ref{eqn-u_n^*}).
\item  If $n=n_{\ell}$, then perform the following:
\begin{enumerate}
\item Compute $\vc{y}_n$ by (\ref{defn-y_n}), and find $j_n^*$ satisfying (\ref{defn-j_n^*}).
\item Compute $\presub{(n)}\wh{\vc{\pi}}_{k}^*$, $k=0,1,\dots,n$, by (\ref{defn-(n)wh{pi}_k^*}).
\item If $\| \presub{(n_{\ell})}\wh{\vc{\pi}}^* - 
\presub{(n_{\ell-1})}\wh{\vc{\pi}}^*\| < \varepsilon$, then stop the iteration; otherwise increment $\ell$ by one and return to step~(a).
\end{enumerate}
\end{enumerate}
\end{enumerate}
\end{algorithm}

\begin{rem}
Equation (\ref{convergence-(n)wh{pi}^*}) leads to
\begin{eqnarray*}
\lim_{n\to\infty} 
\| \presub{(n)}\wh{\vc{\pi}}^* 
- \presub{(n+m)}\wh{\vc{\pi}}^*\| &=& 0
\quad \mbox{for any fixed $m \in \bbN$}.
\end{eqnarray*}
Therefore, our algorithm iterates Step~4 only a finite number of times.
\end{rem}

\begin{rem}\label{rem-Taki16-00}
Step (4.b) computes $\vc{U}_n^*$ by (\ref{eqn-U_n^*}).
The $(i,j)$-th element of $\vc{U}_n^*$ is the expected total sojourn time in state $(n,j)$ before the first visit to $\ol{\bbS}_n$ starting from state $(n,i)$. Thus, $\vc{T}_n^* = (-\vc{U}_n^*)^{-1}$, defined in (\ref{defn-T_n^*}),  is a non-conservative $Q$-matrix that governs the transient transitions of an absorbing Markov chain obtained by observing $\{(X(t),J(t))\}$ when it is in $\bbL_n$ during the first passage time to $\ol{\bbS}_n$ starting from $\bbL_n$. This consideration indicates $\vc{U}_n^* = (-\vc{T}_n^*)^{-1}$ can be  efficiently computed (see \cite[Proposition 1]{Le-Boud91}), provided that $\vc{T}_n^*$ is given.
\end{rem}

\begin{rem}
Generally, our algorithm computes the infinite sum $\sum_{\ell=n+1}^{\infty}\vc{Q}_{k,\ell}\vc{v}_{\ell}$ to obtain $\vc{y}_n$ in (\ref{defn-y_n}). However, 
this infinite sum can be calculated in many practical cases associated with queueing models (as implied by the examples in the next section).
Moreover, if $\vc{Q}$ is an LD-QBD generator, or equivalently, $\vc{Q}_{k,\ell} = \vc{O}$ for $k \in \bbZ_+$ and $|\ell - k| > 1$, then $\vc{y}_n$ is expressed without any infinite sum:
\[
\vc{y}_n = \vc{v}_n + \vc{U}_{n,n}^*\vc{Q}_{n,n+1}\vc{v}_{n+1},
\qquad n \in \bbZ_+.
\]
Furthermore, a noteworthy fact is that computing the infinite sum $\sum_{\ell=n+1}^{\infty}\vc{Q}_{k,\ell}\vc{v}_{\ell}$ is not always necessary even if $\vc{Q}$ is not an LD-QBD generator. To demonstrate this, suppose that we have an explicit expression for $\vc{w}_{k,n}$, $k,n\in\bbZ_+$, such that
\[
\lim_{n\to\infty} \sum_{k=0}^n \vc{\pi}_k\vc{w}_{k,n} = 0,
\quad
\sum_{\ell=n+1}^{\infty}\vc{Q}_{k,\ell}\vc{v}_{\ell} \le \vc{w}_{k,n},
\quad k,n\in\bbZ_+.
\]
It then follows from (\ref{ineqn-r_n(psi_n^*)}) and (\ref{eqn-180318-03}) that
\[
r_n(\wt{\vc{\alpha}}_n)
\le { 1 \over \sum_{\ell=0}^n \vc{\pi}_{\ell}\vc{e} }
\left(
\vc{\pi}_n\vc{\Delta}_n\vc{v}_n
+ \sum_{k=0}^n\vc{\pi}_k \vc{w}_{k,n}
\right)
\to 0\quad \mbox{as $n \to \infty$}.
\]
Thus, we modify Step (4.d.i) as follows: Compute 
\[
\bv{\vc{y}}_n := (\bv{y}_n(j))_{j\in\bbM_n} = \vc{v}_n + \sum_{k=0}^n \vc{U}_{n,k}^*\vc{w}_{k,n},
\]
and find
\[
j_n^*
\in \argmin_{j\in\bbM_n} {\bv{y}_n(j) \over u_n^*(j) }.
\]
Despite this modification, our update algorithm works well.
\end{rem}


\section{Applicability of the proposed algorithm}\label{sec-discussion}

This section demonstrates the applicability of our algorithm. To this end, we consider a BMAP/M/$\infty$ queue and M/M/$s$ retrial queue, respectively, in Sections~\ref{subsec-BMAP-M-Infty} and \ref{subsec-MMss-retrial}. For each model, we present a sufficient condition for Conditions~\ref{cond-01} and \ref{cond-02}, under which our update algorithm works well.

\subsection{BMAP/M/$\infty$ queue}\label{subsec-BMAP-M-Infty}

This subsection considers a BMAP/M/$\infty$ queue. 
The system has an infinite number of servers. Customers arrive at the system according to a batch Markovian arrival process (BMAP) (see, e.g., \cite{Luca91}). Arriving customers are immediately served, and their service times are independent and identically distributed (i.i.d.) with an exponential distribution having mean $\mu^{-1}$.

Let $\{N(t);t \in \bbR_+\}$ denote the counting process of arrivals from the BMAP; that is, $N(t)$ is equal to the total number of arrivals during the time interval $[0,t]$, where $N(0) = 0$. Let $\{J(t);t \in \bbR_+\}$ denote the background Markov chain of the BMAP, which is defined on state space $\bbM=\{1,2,\dots,M\} \subset \bbN$. We assume that the bivariate stochastic process $\{(N(t),J(t));t\in\bbR_+\}$ is a continuous-time Markov chain which follows the transition law given by
\begin{eqnarray*}
\lefteqn{
\PP(N(t+\Delta t) = k,J(t+\Delta t) = j \mid J(t)=i)
}
\quad &&
\nonumber
\\
&=&
\left\{
\begin{array}{l@{~~~}ll}
1 + D_{0,i,i}\Delta t + o(\Delta t), 	& k=0, 	& i = j \in \bbM,
\\
D_{0,i,j}\Delta t + o(\Delta t), 		& k=0,  & i,j \in \bbM,~i \neq j,
\\
D_{k,i,j}\Delta t + o(\Delta t), 		& k \in \bbN, & i,j \in \bbM,
\\
0, 										& \mbox{otherwise},
\end{array}
\right.
\end{eqnarray*}
where $a(t) = o(b(t))$ represents $\lim_{t\to0}a(t)/b(t) = 0$. Thus, the BMAP is characterized by $\{\vc{D}_n;n\in\bbZ_+\}$, where $\vc{D}_n=(D_{n,i,j})_{i,j\in\bbM}$ for $n \in \bbZ_+$. Moreover, $\vc{D}:=\sum_{n\in\bbZ_+}\vc{D}_n$
 is the generator of the background Markov chain $\{J(t);t\in\bbR_+\}$. As usual, we assume that $\vc{D}$ is irreducible and
\[
\vc{D}\vc{e} \ge \vc{0},\neq \vc{0}.
\]

Let $X(t)$, $t \in \bbR_+$, denote the number of customers in the system at time $t$. It then follows that $\{(X(t),J(t));t\in\bbR_+\}$ is a continuous-time Markov chain on state space $\bbS:=\bbZ_+ \times \bbM$ with generator $\vc{Q}$ given by
\begin{equation}
\vc{Q} = 
\bordermatrix{
        & \bbL_0 & \bbL_1    & \bbL_2    & \bbL_3   & \cdots
\cr
\bbL_0 		& 
\vc{D}_0  	& 
\vc{D}_1 	& 
\vc{D}_2 	& 
\vc{D}_3 	&  
\cdots
\cr
\bbL_1 		&
\mu\vc{I}  	& 
\vc{D}_0-\mu\vc{I}  & 
\vc{D}_1 	& 
\vc{D}_2 	&    
\cdots
\cr
\bbL_2 			& 
\vc{O}  		& 
2\mu\vc{I}		&
\vc{D}_0-2\mu\vc{I} & 
\vc{D}_1 		&  
\cdots
\cr
\bbL_3 & 
\vc{O}  		& 
\vc{O}  		& 
3\mu\vc{I}		&
\vc{D}_0-3\mu\vc{I} & 
\cdots
\cr
~\vdots  	& 
\vdots     		& 
\vdots     		&  
\vdots    		& 
\vdots    		& 
\ddots
},
\label{defn-Q-BMAP-G-infty}
\end{equation}
where $\bbL_k = \{k\} \times \bbM$ (i.e., $M_k = M$) for all $k\in\bbZ_+$ and
\begin{eqnarray}
\vc{Q}_{k,\ell}
=
\left\{
\begin{array}{l@{~~~}ll}
\vc{D}_{\ell-k}, 		& k \in \bbZ_+, & \ell = k+1,k+2,\dots,
\\
\vc{D}_0 - k\mu\vc{I},  & k \in \bbN,  	&  \ell=k,
\\
k\mu\vc{I},  			& k \in \bbN,   & \ell=k-1.
\end{array}
\right.
\label{eqn-Q-BMAP-M-infty}
\end{eqnarray}

We now suppose that, for some $C>0$,
\begin{equation}
\sum_{k=1}^{\infty}(k+\re)\log(k + \re) \vc{D}_k \vc{e} < C\vc{e},
\label{cond-BMAP-M-infty}
\end{equation}
and let
\begin{equation*}
\vc{v}_k = \log(k + \re) \vc{e},\qquad k \in \bbZ_+,
\end{equation*}
where ``$\re$" denotes Napier's constant. Clearly, $\sum_{k=1}^{\infty}\log(k + \re) \vc{D}_k \vc{e} < C\vc{e}$ due to (\ref{cond-BMAP-M-infty}). Thus, Condition \ref{cond-01} holds for generator $\vc{Q}$ in (\ref{defn-Q-BMAP-G-infty}) (see \cite[Lemma~1]{Yaji16}).

It remains to verify that Condition \ref{cond-02} holds. From (\ref{defn-Delta_n}) and (\ref{defn-Q-BMAP-G-infty}), we have $\Delta_k(i,i) = |q(k,i;k,i)| = k\mu + |D_{0,i,i}|$ and thus Condition \ref{cond-02} is reduced to
\begin{equation}
\sum_{(k,i)\in\bbZ_+\times\bbM} \pi(k,i) k \log(k + \re) < \infty.
\label{ineqn-pi*v-03}
\end{equation}
Therefore, we show that (\ref{ineqn-pi*v-03}) holds.

We begin with the following lemma.
\begin{lem}\label{lem-BMAP-M-infty-subexp}
Let $V$ denote a function on $\bbR_+$ such that
\begin{equation}
V(x) = (x + \re) \log(x + \re),\qquad x \in \bbR_+.
\label{defn-V(x)}
\end{equation}
If (\ref{cond-BMAP-M-infty}) holds, then there exist some $K \in \bbZ_+$ and $\theta > 0$ such that
\begin{eqnarray}
\sum_{\ell=0}^{\infty}\vc{Q}_{k,\ell} V(\ell)\vc{e}
&\le& -\theta V(k) \vc{e} \le -\vc{e}\quad \mbox{for all $k \ge K+1$}.
\label{ineqn-sum-Q*v-subexp-k>=K}
\end{eqnarray}

\end{lem}

\begin{proof}
Because $\lim_{x\to\infty}V(x) = \infty$, it suffices to prove that
\begin{eqnarray}
\limsup_{k\to\infty}
{1 \over V(k)}\sum_{\ell=0}^{\infty}\vc{Q}_{k,\ell} V(\ell)\vc{e} 
&\le& -\mu \vc{e}.
\label{limsup-sum-Q*v/V(k)}
\end{eqnarray}
It follows from (\ref{eqn-Q-BMAP-M-infty}) that, for $k\in\bbN$,
\begin{eqnarray}
{1 \over V(k)}
\sum_{\ell=0}^{\infty}\vc{Q}_{k,\ell} V(\ell)\vc{e}
&=& 
\left\{
-\mu k \left( 1 - {V(k-1) \over V(k)} \right) \vc{e}
+ \sum_{\ell=0}^{\infty} { V(k+\ell) \over V(k) }\vc{D}_{\ell}
\vc{e}
\right\}.\quad~~~~
\label{eqn-Qv-subexp}
\end{eqnarray}
Furthermore, $V$ is differentiable and convex. Thus, we have
\[
V(k) \ge V(k-1) + V'(k-1),\qquad k \ge 1.
\]
Using this inequality and (\ref{defn-V(x)}), we obtain
\begin{eqnarray}
\liminf_{k\to\infty}
k\left( 1 - {V(k-1) \over V(k)} \right)
&\ge& \liminf_{k\to\infty} k {V'(k-1) \over V(k)}
\nonumber
\\
&=& \lim_{k\to\infty}
{k \over k + \re}{\log (k - 1 + \re) + 1 \over \log (k + \re)}
 = 1.
\label{liminf-k*V(k-1)/V(k)}
\end{eqnarray}
Applying (\ref{liminf-k*V(k-1)/V(k)}) to (\ref{eqn-Qv-subexp}) yields
\begin{eqnarray*}
\limsup_{k\to\infty}
{1 \over V(k)}
\sum_{\ell=0}^{\infty}\vc{Q}_{k,\ell} V(\ell)\vc{e}
\le -\mu\vc{e}
+ \limsup_{k\to\infty}
\sum_{\ell=0}^{\infty} { V(k+\ell) \over V(k) }\vc{D}_{\ell}
\vc{e},
\end{eqnarray*}
and therefore (\ref{limsup-sum-Q*v/V(k)}) holds if
\begin{equation}
\lim_{k\to\infty}
\sum_{\ell=0}^{\infty} { V(k+\ell) \over V(k) }\vc{D}_{\ell}\vc{e}
= \vc{0}.
\label{lim-sum-V(k+l)/V(k)*D(l)}
\end{equation}
Consequently, our goal is to prove (\ref{lim-sum-V(k+l)/V(k)*D(l)}).

We note that $V \ge 1$ is log-concave, which implies the following: For any $x,y \in \bbR_+$ such that $x+y>0$,
\begin{align*}
\log V(x) &\ge {y \over x+y} \log V(0) + {x \over x+y} \log V(x+y),
\\
\log V(y) &\ge {x \over x+y} \log V(0) + {y \over x+y} \log V(x+y).
\end{align*}
These inequalities yield
\[
\log V(x) + \log V(y)
\ge \log V(0) + \log V(x+y) \ge \log V(x+y),
\]
which leads to
\begin{equation}
V(x+y) \le V(x)V(y),\qquad x,y \in \bbR_+.
\label{indeqn-V(x+y)}
\end{equation}
Using (\ref{indeqn-V(x+y)}) and (\ref{defn-V(x)}), we obtain, for all $k \in \bbZ_+$,
\begin{eqnarray*}
\sum_{\ell=1}^{\infty} { V(k+\ell) \over V(k) }\vc{D}_{\ell}\vc{e}
\le
\sum_{\ell=1}^{\infty}  V(\ell)\vc{D}_{\ell}\vc{e}
= \sum_{\ell=1}^{\infty} (\ell + \re) \log (\ell + \re)\vc{D}_{\ell}\vc{e}
< C\vc{e},
\end{eqnarray*}
where the last inequality is due to (\ref{cond-BMAP-M-infty}). Thus, by the dominated convergence theorem and (\ref{defn-V(x)}), we obtain
\begin{equation*}
\lim_{k\to\infty}
\sum_{\ell=0}^{\infty} { V(k+\ell) \over V(k) }\vc{D}_{\ell}\vc{e}
=
\vc{D}_0\vc{e}
+
\sum_{\ell=1}^{\infty} 
\lim_{k\to\infty}{ V(k+\ell) \over V(k) }\vc{D}_{\ell}\vc{e}
= \sum_{\ell=0}^{\infty}\vc{D}_{\ell}\vc{e}
= \vc{0},
\end{equation*}
which shows that (\ref{lim-sum-V(k+l)/V(k)*D(l)}) holds. \qed
\end{proof}

\medskip

Let $\wt{\vc{v}}:=(\wt{v}(k,i))_{(k,i) \in \bbZ_+ \times \bbM}$ and $\wt{\vc{f}}:=(\wt{f}(k,i))_{(k,i) \in \bbZ_+ \times \bbM}$ denote column vectors such that
\begin{eqnarray}
\wt{v}(k,i) &=& V(k) = (k + \re)\log (k + \re), ~~\qquad k \in \bbZ_+, \quad~~ i \in \bbM,
\nonumber
\\
\wt{f}(k,i) &=& 
\left\{
\begin{array}{l@{~~~}ll}
1, 			 &  0 \le k \le K,     & i \in \bbM,
\\
\theta V(k) = \theta (k + \re)\log (k + \re), &  k \ge K+1, & i \in \bbM,
\end{array}
\right.
\label{defn-f-Qv-BMAP-M-infty-heavy}
\end{eqnarray}
where $K \in \bbZ_+$ and $\theta > 0$ satisfying (\ref{ineqn-sum-Q*v-subexp-k>=K}). It then follows from Lemma~\ref{lem-BMAP-M-infty-subexp} that, for some $\wt{b} > 0$,
\begin{equation*}
\vc{Q}\wt{\vc{v}} \le -\wt{\vc{f}} + \wt{b}\vc{1}_{\bbZ_K \times \bbM},
\end{equation*}
which yields $\vc{\pi}\wt{\vc{f}} < \wt{b}$. Combining this inequality and (\ref{defn-f-Qv-BMAP-M-infty-heavy}) results in (\ref{ineqn-pi*v-03}). 
We have confirmed that Condition \ref{cond-02} is satisfied. As a result,
our algorithm is always applicable to BMAP/M/$\infty$ queues satisfying (\ref{cond-BMAP-M-infty}).

\subsection{M/M/$s$ retrial queue}\label{subsec-MMss-retrial}

In this subsection, we consider an M/M/$s$ retrial queue (which is sometimes called an M/M/$s$/$s$ retrial queue).
The system has $s$ ($s \in \bbN$) servers but no {\it real} waiting room.
{\it Primary customers} (which originate from the exterior) arrive to the system according to a Poisson process with rate $\lambda
\in (0,\infty)$. If an arriving primary customer finds an idle server, then the customer occupies the server, otherwise it joins the {\it orbit} (i.e., the {\it virtual} waiting room). Customers in the orbit are referred to as {\it retrial customers}. Each retrial customer stays in the orbit for an exponentially distributed time with mean $\eta^{-1} \in (0,\infty)$, independently of all the other events.
After the sojourn in the orbit, a retrial customer tries to occupy one of idle servers.  If such a retrial customer finds no idle servers, then it goes back to the orbit; that is, becomes a retrial customer again. We assume that the service times of primary and retrial customers are i.i.d.\ with an
exponential distribution having mean $\mu^{-1} \in (0,\infty)$.

Let $X(t)$, $t\in\bbR_+$, denote the number of customers in the orbit
at time $t$. Let $J(t)$, $t\in\bbR_+$, denote the number of busy
servers at time $t$. The stochastic process $\{(X(t),J(t));t\in\bbR_+\}$ is a level-dependent
quasi-birth-and-death process (LD-QBD) on state space $\bbS:= \bbZ_+
\times \bbZ_s$ with generator $\vc{Q}$ given by
\begin{equation}
\vc{Q} =
\bordermatrix{
               & \bbL_0 &  \bbL_1  &  \bbL_2       &  \bbL_3       & \cdots
\cr
\bbL_0 &
\vc{Q}_{0,0} &
\vc{Q}_{0,1} &
\vc{O} &
\vc{O} &
\cdots
\cr
\bbL_1 &
\vc{Q}_{1,0} &
\vc{Q}_{1,1} &
\vc{Q}_{1,2} &
\vc{O} &
\cdots
\cr
\bbL_2 &
\vc{O}&
\vc{Q}_{2,1} &
\vc{Q}_{2,2} &
\vc{Q}_{2,3} &
\cdots
\cr
\bbL_3
& \vc{O}					
&
\vc{O}					
&
\vc{Q}_{3,2}&
\vc{Q}_{3,3} &
\ddots
\cr
~~\vdots
& \vdots					
&
\vdots					&
\vdots					&
\ddots					&
\ddots
},
\label{eqn-Q-LD-QBD}
\end{equation}
where $\bbL_k = \{k\} \times \bbZ_s$ for $k \in \bbZ_+$, and where
\begin{align}
&&&&
\vc{Q}_{k,k-1}   
&=  
\left (
\begin{array}{llllll}
0 \  & k \eta  \  & 0 & \cdots &  0 
\\
0 & 0 & k \eta \  &  \ddots &  \vdots 
\\
\vdots &  & \ddots & \ddots &   0 
\\
\vdots &   &  & 0 \  & k \eta  
\\
0 & \cdots & \cdots & 0 & 0 
\\
\end{array}
\right ), & k &\in \bbN,&&&&
\label{eqn-Q_{k,k-1}}
\\
&&&&
\vc{Q}_{k,k+1}  &=  
\left (
\begin{array}{cccccc}
0      & 0      & \cdots & 0      & 0 
\\
0      & 0      & \cdots & 0      & 0 
\\
\vdots & \vdots & \ddots & \vdots & \vdots 
\\
0      & 0      & \cdots & 0      & 0 
\\
0      & 0      & \cdots & 0      & \lambda  
\\
\end{array}
\right ), & k &\in \bbZ_+,&&&&
\label{eqn-Q_{k,k+1}}
\end{align}
and
\begin{eqnarray}
\vc{Q}_{k,k}  
&=& 
\left (
\begin{array}{cccccc}
-\psi_{k,0} \ & \lambda\ & 0 & \cdots   &  \cdots & 0 
\\
\mu   & -\psi_{k,1}  & \lambda  & \ddots   &   & \vdots 
\\
0  & 2 \mu \ & -\psi_{k,2} \ & \ddots &  \ddots & \vdots 
\\
\vdots & \ddots & \ddots & \ddots & \ddots & 0 
\\
\vdots & & \ddots & \ddots \quad & -\psi_{k,s-1} \quad & \lambda 
\\
0 & \cdots & \cdots & 0 &  s \mu \ & -\psi_{k,s}
\\
\end{array}
\right ),\quad k \in \bbZ_+,\qquad
\label{eqn-Q_{k,k}}
\end{eqnarray}
with
\begin{align*}
&&&&
\psi_{k,i}
&= \lambda + i \mu + k\eta, & k &\in \bbZ_+,\ i \in \bbZ_{s-1},
&&&&
\\
&&&&
\psi_{k,s}
&= \lambda + s \mu, & k &\in \bbZ_+.
&&&&
\end{align*}

We now assume that the stability condition $\rho := \lambda / (s\mu) < 1$ holds. It then follows that 
the LD-QBD $\{(X(t),J(t))\}$ is ergodic (see, e.g., \cite[Section~2.2]{Fali97}) and thus has the unique stationary distribution vector $\vc{\pi}=(\pi(k,i))_{(k,i)\in\bbS}$. Under this stability condition, we show that Conditions~\ref{cond-01} and \ref{cond-02} are satisfied, which requires the following proposition.
\begin{prop}[{}{\cite[Lemma~4.1]{Masu17-JORSJ}}]\label{lem-MMC}
Suppose that $\rho = \lambda /
(s\mu) < 1$. For $k \in \bbZ_+$, let $\vc{v}_k=(v(k,j))_{j\in\bbZ_s}$ be given by
\begin{eqnarray}
v(k,i)
&=& \left\{
\begin{array}{l@{~~~}ll}
\alpha^k/c, 		& k \in \bbZ_+,&  i \in \bbZ_{s-1},
\\
\alpha^k/(c\gamma),	& k \in \bbZ_+,& i = s,
\end{array}
\right.
\label{eqn-v(k)-retrial}
\end{eqnarray}
where $\alpha$, $\gamma$, and $c$ are positive constants such that
\begin{eqnarray}
1 &<& \alpha < \rho^{-1},
\label{defn-theta}
\\
\alpha^{-1}
&<& \gamma
< 1 - \rho(\alpha-1),
\nonumber
\\
c &=& s\mu \left[ 1 - \rho(\alpha-1) - \gamma \right].
\nonumber
\end{eqnarray}
Furthermore, let
\begin{eqnarray*}
b &=& \max_{k \in \bbZ_{K}}\alpha^k
\left[1 -
c^{-1}\{ k\eta ( 1 - \gamma^{-1}\alpha^{-1})  + \lambda ( 1 - \gamma^{-1} ) \}
\right] \vmax 0,
\nonumber
\\
K &=&
\left\lceil
{ c + \lambda ( \gamma^{-1} - 1) \over \eta ( 1 - \gamma^{-1}\alpha^{-1}) }
\right\rceil \vmax 1 - 1,
\nonumber
\end{eqnarray*}
where $x \vmax y = \max(x,y)$. Under these conditions, the generator
$\vc{Q}$ of the LD-QBD, characterized by (\ref{eqn-Q-LD-QBD})--(\ref{eqn-Q_{k,k}}), satisfies
\begin{equation*}
\vc{Q}\vc{v}
\le - c\vc{v} + b \vc{1}_{\bbS_{K}}.
\end{equation*}
\end{prop}

We note that $c\vc{v} \ge \vc{e}.$
Proposition~\ref{lem-MMC} thus shows that Condition~\ref{cond-01} is satisfied.
Moreover, Theorem~1 in \cite{KimJeri12} states that, for a certain constant $c_0 > 0$,
\begin{equation}
\pi(k,i) \simhm{k} {c_0 \over i!} \left(\eta \over \mu \right)^i
k^{-s+i+\lambda/(s\eta)} \rho^k,\qquad i \in \bbZ_s,
\label{asymp-pi(k,i)}
\end{equation}
where $a_1(x) \simhm{x} a_2(x)$ represents $\lim_{x\to\infty}a_1(x)/a_2(x) = 1$. Combining (\ref{eqn-v(k)-retrial})--(\ref{asymp-pi(k,i)}) yields
\[
\sum_{(k,i) \in \bbS} \pi(k,i) k v(k,i) < \infty,
\]
which implies that Condition~\ref{cond-02} is satisfied.
Consequently, our algorithm is always applicable to stable M/M/$s$ retrial queues.

\section{Concluding Remarks}\label{sec-remarks}

This paper has presented a sequential update algorithm for computing the stationary distribution vector in continuous-time upper block-Hessenberg Markov chains. The algorithm stops after finitely many iterations if Conditions~\ref{cond-01} and \ref{cond-02} are satisfied. These conditions hold in any stable M/M/$s$ retrial queue and the BMAP/M/$\infty$ queues satisfying the mild condition (\ref{cond-BMAP-M-infty}). Furthermore, the algorithm would be applicable (under some mild conditions) to MAP/PH/$s$ retrial queues, BMAP/PH/$\infty$ queues, and their variants.

\appendix

\section{Proof of Proposition~\ref{lem-pi_n-Lambda_n}}\label{appen-ineqn-pi-(U_n^*)^{-1}}

Let $\vc{T}_n^*$, $n \in \bbZ_+$, denote
\begin{equation}
\vc{T}_n^*
= 
\left\{
\begin{array}{l@{~~~}l}
\vc{Q}_{0,0}, & n=0,
\\

\vc{Q}_{n,n} 
+ \dm\sum_{\ell=0}^{n-1} \vc{U}_{n,\ell} \vc{Q}_{\ell,n},
& n \in \bbN.
\end{array}
\right.
\label{defn-T_n^*}
\end{equation}
It then follows from (\ref{defn-U_k^*}), (\ref{defn-Delta_n}), and (\ref{defn-T_n^*}) that
 \begin{equation}
\vc{\pi}_n\vc{\Delta}_n \ge
\vc{\pi}_n(-\vc{Q}_{n,n}) \ge \vc{\pi}_n(-\vc{T}_n^*) 
= \vc{\pi}_n(\vc{U}_n^*)^{-1},
\qquad n \in \bbZ_+.
\label{eqn-180909-01}
\end{equation}
It also follows from (\ref{eqn-pi_l-pi_k*U_{k,l}}), (\ref{defn-T_n^*}), and $\sum_{\ell=0}^{\infty} \vc{\pi}_{\ell} \vc{Q}_{\ell,n} = \vc{0}$ ($n\in\bbZ_+$) that
\begin{eqnarray}
\vc{\pi}_n(-\vc{T}_n^*)
&=& -\vc{\pi}_n \vc{Q}_{n,n} 
-\vc{\pi}_n \dm\sum_{\ell=0}^{n-1} \vc{U}_{n,\ell} \vc{Q}_{\ell,n}
\nonumber
\\
&=& - \sum_{\ell=0}^n \vc{\pi}_{\ell} \vc{Q}_{\ell,n}
= \sum_{\ell=n+1}^{\infty} \vc{\pi}_{\ell} \vc{Q}_{\ell,n}
 \ge \vc{0},\neq \vc{0},\qquad n \in \bbZ_+.
\label{eqn-180909-02}
\end{eqnarray}
Combining (\ref{eqn-180909-01}) and (\ref{eqn-180909-02}) yields (\ref{ineqn-pi-(U_n^*)^{-1}}).
The proof has been completed.



\section*{Acknowledgments}
The author thanks Mr. Masatoshi Kimura and Dr. Tetsuya Takine for providing the counterexample presented in Section~\ref{subsec-example}. 
The author also thanks an anonymous referee for his/her valuable comments that helped to improve the paper.

%
%
%
\bibliographystyle{plain} 

\end{document}